\documentclass[11pt,twoside,reqno,centertags]{amsart}
\usepackage[oldstylenums]{kpfonts}
\usepackage{hyperref}
\hypersetup{
    colorlinks,
    citecolor=blue,
    filecolor=blue,
    linkcolor=blue,
    urlcolor=blue
}
\usepackage{geometry}

\usepackage[utf8]{inputenc}
\usepackage[english]{babel}

\newcommand{\rd}{\mathrm d}
\newcommand{\cd}{\,\mathrm d}
\usepackage{amsmath, amsfonts, amssymb}
\usepackage{amsthm}
\newtheorem{theorem}{Theorem}
\newtheorem{lemma}[theorem]{Lemma}

\theoremstyle{remark}
\newtheorem{remark}[theorem]{Remark}
\theoremstyle{definition}
\newtheorem{definition}[theorem]{Definition}

\numberwithin{theorem}{section}
\numberwithin{equation}{section}

\DeclareMathOperator{\supp}{supp}
\DeclareMathOperator{\grad}{grad}
\DeclareMathOperator*{\esslimsup}{ess\;lim\;sup}
\DeclareMathOperator*{\esssup}{ess\;sup}

\DeclareMathOperator{\Div}{div}
\newcommand{\energy}{\mathcal W}
\newcommand{\entropy}{\mathcal E}
\newcommand{\entropyproduction}{D\mathcal E}
\newcommand {\R} {\mathbb{R}}

\date{}

\title[Free energy Fokker-Planck]{Nonlinear Fokker-Planck equations with 
reaction as gradient flows of the free energy}
\author[S.~Kondratyev]{Stanislav Kondratyev}
\address[S.~Kondratyev]{CMUC, Department of
Mathematics, University of Coimbra, 3001-501 Coimbra, Portugal}{}
\email{kondratyev@mat.uc.pt}

\author[D.~Vorotnikov]{Dmitry Vorotnikov}
\address[D.~Vorotnikov]{CMUC, Department of
Mathematics, University of Coimbra, 3001-501 Coimbra, Portugal}{}
\email{mitvorot@mat.uc.pt}

\begin{document}
\maketitle
\begin{abstract}
We interpret a class of nonlinear Fokker-Planck equations with reaction as 
gradient flows over the space of Radon measures equipped with the recently 
introduced Hel\-linger-Kantorovich distance. The driving entropy of the gradient flow is not assumed to be geodesically convex or semi-convex.  We prove new generalized dissipation inequalities, which allow us to control the 
relative entropy by its production.  We establish the 
entropic exponential convergence of the trajectories of the flow to the 
equilibrium. Along with other applications, this result has an ecological interpretation as a 
trend to the ideal free distribution for a class of fitness-driven models of 
population dynamics. Our existence theorem for weak solutions under mild 
assumptions on the nonlinearity is new even in the absence of the reaction 
term.
\end{abstract}

\vspace{10pt}


\vspace{10pt}

\textbf{MSC [2010] 26D10, 35Q84, 49Q20, 58B20}

\section{Introduction}

\subsection{Setting}
\label{ss:s}

Let~$\Omega$ be an open connected bounded domain in~$\mathbb R^d$ with 
sufficiently smooth boundary and let~$\nu$ be the outward unit normal 
along~$\partial \Omega$.  We are interested in nonnegative solutions of
\begin{align}
\partial_t u & = -\Div (u \nabla f) + fu, & (x, t) & \in \Omega \times (0, 
\infty),
\label{eq:p1}
\\
u \frac{\partial f}{\partial \nu} & = 0, & (x, t) & \in \partial \Omega \times 
(0, \infty),
\label{eq:p2}
\\
u & = u^0, & (x, t) & \in \Omega \times {0}.
\label{eq:p3}
\end{align}
Here $u$ is the unknown function, $f = f(x, u(x,t))$ is a known nonlinear 
function of $x$ and $u$, equation~\eqref{eq:p2} is the no-flux boundary 
condition and the initial data~$u^0$ are nonnegative. We refer to Section 
\ref{ss:mb} for the motivation and background.

When considering problem~\eqref{eq:p1}--\eqref{eq:p3}, we always make the 
following assumptions concerning the function~$f \colon \Omega \times (0, 
\infty) \to \mathbb R$:
\begin{gather}
f \in
C^2( \overline\Omega \times (0, \infty))
\cap L^1_\text{loc}(\overline \Omega \times [0, \infty))
\label{eq:f1a}
\\
uf, uf_x
\in C(\overline \Omega \times [0, \infty))
\label{eq:f1b}
\\
f_u < 0,
\label{eq:f2}
\\
\limsup_{u \to \infty} f(x, u) < 0 \quad \forall x \in \overline \Omega,
\label{eq:f3}
\\
\liminf_{u \to +0} f(x, u) > 0 \quad \forall x \in \overline \Omega
,
\label{eq:f4}
\\
|f(x, u)| + u |f_u(x, u)| + u |f_{xu}(x, u)| \le g(u)
\quad
\text{a.~a. } u > 0; \ g \in L^1_\text{loc}[0, \infty)
,
\label{eq:f5}
\\
(uf_x)\big|_{u = 0} = 0
.
\label{eq:f1d}
\end{gather}

When needed, we also assume that
\begin{align}
\text{either $f_x = 0$ for large $u$} & \text{\quad or\quad} \lim_{u \to 
\infty} f(x, u) = - \infty \ \forall x \in \overline \Omega
\label{eq:large-u}
\\
\text{either $f_x = 0$ for small $u$} & \text{\quad or\quad} \lim_{u \to +0} 
f(x, u) = \infty \ \forall x \in \overline \Omega
\label{eq:small-u}
\end{align}

\begin{remark}
We make comfortable assumptions about the smoothness of~$f$.  We do not insist 
that~$f$ should be defined for $u = 0$ so as not to exclude the interesting 
cases such as $f = -(\log u+V(x))$ (which corresponds to the linear 
Fokker-Planck equation, cf. \cite{F05,JKO}) and $f=u^\alpha-1$, $-1<\alpha<0$, 
(the fast diffusion, cf. \cite{Vaz07}).  However, we assume in~\eqref{eq:f1b} 
that the functions~$uf$ and~$uf_x$ admit continuous extensions to $\overline 
\Omega \times [0, \infty)$.  This ensures that the terms in~\eqref{eq:p1} make 
sense.  Moreover, we assume~\eqref{eq:f1d} to avoid certain complications with 
the entropy production to be defined below.
\end{remark}

\begin{remark}
Assumption~\eqref{eq:f2} is essential, it ensures the parabolicity 
of~\eqref{eq:p1}.  The equation may become degenerate or singular only if $u = 
0$ or $u$ is large.  The latter does not bother us as we only consider bounded 
solutions in what follows.
\end{remark}

\begin{remark}
Assumptions~\eqref{eq:f3}, \eqref{eq:f4} ensure the existence of a positive 
equilibrium, see below.
\end{remark}

\begin{remark}
Estimate~\eqref{eq:f5} ensures that the entropy and energy of the equation are 
well-defined and well-behaved.  Note that at least some restrictions on the 
growth of~$f_u$ as $u \to 0$ are inevitable, as the related very fast 
diffusion equation is known to behave abnormally \cite{Vaz-igolki}.
\end{remark}

\begin{remark}
Conditions~\eqref{eq:large-u} and~\eqref{eq:small-u} are convenient technical 
assumptions needed for $L^\infty$-bounds (hence for the existence theorem) and 
for controlling the energy for large~$u$ in the proof of Theorem~\ref{th:eep}.  
However, they are not necessary everywhere, so we explicitely mention them 
when the need arises.
\end{remark}

\begin{remark} The results of the paper remain valid if $\Omega$ is the periodic box $\mathbb T^d$. 
\end{remark}

It follows from~\eqref{eq:f2}--\eqref{eq:f4} that for any~$x \in \overline 
\Omega$ there exists a unique $m(x) > 0$ such that
\begin{equation*}
f(x, m(x)) = 0
.
\end{equation*}
Clearly, $m \in C^2(\overline \Omega)$.  It is a stationary solution 
of~\eqref{eq:p1}, \eqref{eq:p2}.  As we will see, all non-zero solutions of 
the problem converge to~$m$.

\subsection{Energy and entropy}
\label{ss:ee}

Now we will introduce the energy and entropy functionals for 
equation~\eqref{eq:p1} as well as the notion of weak solution.

Put
\begin{equation*}
\Phi(x, u) = - \int_0^u \xi f_u(x, \xi) \cd \xi,
\quad
\Psi(x, u) = \int_0^u \Phi(x, \xi) \cd \xi
.
\end{equation*}
%By the above, $\Phi, \Psi \in C^1(\overline \Omega \times [0, \infty))$.  
It is easy to see that
\begin{equation*}
\Phi(x, 0) = \Psi(x, 0) = 0,
\ \Phi_u = - u f_u,
\ \Phi_x = - \int_0^u \xi f_{xu}(x, \xi) \cd \xi,
\ \Psi_u = \Phi.
\end{equation*}
Observe that both~$\Phi$ and~$\Psi$ are nonnegative and strictly increase with 
respect to~$u$.

Note that if $u$ is a nonnegative function of~$x$ and possibly of~$t$, an 
$L^\infty$-bound on~$u$ is translated into an $L^\infty$-bound on~$\Phi(\cdot, 
u(\cdot))$, i.~e., the superposition operator associated with~$\Phi$ is 
$L^\infty$-bounded.  The same is true of~$\Psi$.

Let $u$ be a classical solution of~\eqref{eq:p1}--\eqref{eq:p3}.  
Equation~\eqref{eq:p1} can be cast in the equivalent form
\begin{equation}
\label{eq:p1a}
\partial_t u = \Delta \Phi - \Div(\Phi_x + uf_x) + uf
,
\end{equation}
where we write~$\Phi$ for~$\Phi(x, u(x, t))$, etc.  Multiplying by $\Phi(x, 
u(x, t))$ and integrating over~$\Omega$, we obtain
\begin{equation}
\label{eq:energy-id}
\partial_t
\int_\Omega \Psi \cd x
= - \int_{\Omega} | \nabla \Phi |^2 \cd x
+ \int_{\Omega} (\Phi_x + uf_x) \cdot \nabla \Phi \cd x
+ \int_{\Omega} uf\Phi \cd x
.
\end{equation}
We call the functional
\begin{equation*}
\energy(u) = \int_\Omega \Psi(x, u(x)) \cd x
\end{equation*}
the \emph{energy} of problem \eqref{eq:p1}--\eqref{eq:p3} and 
equation~\eqref{eq:energy-id}, the \emph{energy identity}.  Thus, \emph{any 
classical solution of~\eqref{eq:p1}--\eqref{eq:p3} satisfies the energy 
identity~\eqref{eq:energy-id}}.
%We have thus proved
%\begin{lemma}
%\label{lem:energy-class}
%Any classical solution of~\eqref{eq:p1}--\eqref{eq:p3} satisfies the energy 
%identity~\eqref{eq:energy-id}.
%\end{lemma}

For our purposes, the energy identity is useful because it allows us to 
control the integral $\iint_{Q_T} |\nabla \Phi |^2 \cd x \cd t$.  In 
particular, we can define the weak solution of \eqref{eq:p1}--\eqref{eq:p3} in 
a class of functions~$u$ such that $\Phi(\cdot, u(\cdot)) \in L^2(0, T; 
H^1(\Omega))$.  It is easier to exploit this assumption in the case of 
equation~\eqref{eq:p1a}.  Thus, we define the weak solution as follows:

\begin{definition}
Let $u^0 \in L^\infty(\Omega)$.  A function~$u \in L^\infty(Q_T)$ is called a 
\emph{weak solution} of~\eqref{eq:p1}--\eqref{eq:p3} on $[0, T]$ if 
$\Phi(\cdot, u(\cdot)) \in L^2(0, T; H^1(\Omega))$ and
\begin{equation}
\label{eq:def1}
\int_0^T
\int_\Omega
(u \partial_t \varphi + ( - \nabla \Phi + \Phi_x + u f_x) \cdot \nabla \varphi 
+ f u \varphi) \cd x \cd t
=
\int_\Omega u^0(x) \varphi(x, 0) \cd x
\end{equation}
for any function $\varphi \in C^1(\overline \Omega \times [0, T])$ such that 
$\varphi(x, T) = 0$.
A function $u \in L^\infty_\text{loc}([0, \infty); L^\infty(\Omega))$ is 
called a \emph{weak solution} of \eqref{eq:p1}--\eqref{eq:p3} on $[0, \infty)$ 
if for any~$T > 0$ it is a weak solution on~$[0, T]$.
\end{definition}

Now, let us address the entropy of the problem.  Define
\begin{equation*}
E(x, u) = - \int_{m(x)}^u f(x, \xi) \cd \xi
.
\end{equation*}
It follows from~\eqref{eq:f5} that~$E$ is well-defined and continuous on 
$\overline \Omega \times [0, \infty)$.  As $f$ decreases with respect to~$u$ 
and $f(x, m(x)) = 0$, it is clear that $E \ge 0$ and $E(x, u) = 0$ if and only if 
$u = m(x)$.  The relative \emph{entropy} of equation~\eqref{eq:p1} is the 
functional
\begin{equation} \label{eq:entr}
\entropy (u) = \int_\Omega E(x, u(x)) \cd x
.
\end{equation}
Observe that it is well-defined at least for $u \in L^\infty_+(\Omega)$ as the 
superposition operator $u \mapsto E(\cdot, u(\cdot))$ is bounded in the spaces 
$L^\infty_+ \to L^\infty_+$.

A straightforward computation shows that for a positive classical solution 
of~\eqref{eq:p1}--\eqref{eq:p3} we have
\begin{equation}
\label{eq:entropy-diss-class}
\partial_t \entropy(u) = - \int_\Omega u(f^2 + | \nabla f |^2) \cd x
.
\end{equation}
Equation~\eqref{eq:entropy-diss-class} is called the \emph{entropy dissipation 
identity} and the integral on the right-hand side 
of~\eqref{eq:entropy-diss-class} is called the~\emph{entropy production}.  
However, the term $\int_\Omega u|\nabla f|^2 \cd x$ may make no sense for 
vanishing or non-smooth~$u$.  In order to generalise the definition of the 
entropy production, we use the identity
\begin{equation*}
u | \nabla f|^2 = \frac{1}{u} | -\nabla \Phi + \Phi_x + uf_x |^2 \quad (u > 0)
.
\end{equation*}
Given a function $u \in L_+^\infty(\Omega)$ such that $\Phi(\cdot, u(\cdot)) 
\in H^1(\Omega)$, the right-hand side of the last identity is a nonnegative 
measurable function on $[u > 0]$, so we can define the entropy production for 
such functions by the formula
\begin{equation*}
\entropyproduction(u) = \int_\Omega uf^2 \cd x
+ \int_{[ u > 0]} \frac 1 u | - \nabla \Phi + \Phi_x + uf_x |^2 \cd x
,
\end{equation*}
where the second integral on the right-hand side may be infinite. Thus, we see 
that \emph{any positive classical solution of~\eqref{eq:p1}--\eqref{eq:p3} 
satisfies the entropy dissipation identity}
\begin{equation}
\label{eq:entropy-id}
\partial_t \entropy(u) = - \entropyproduction(u)
.
\end{equation}

As usual, in the case of weak solutions we establish not the 
identities~\eqref{eq:energy-id} and \eqref{eq:entropy-id} but rather 
corresponding inequalities, viz.\ the \emph{energy inequality}
\begin{equation}
\label{eq:energy}
\partial_t
\energy(u)
\le
\int_{\Omega}
\big (
- | \nabla \Phi |^2
+ (\Phi_x + uf_x) \cdot \nabla \Phi
+ uf\Phi
\big )
\cd x
\end{equation}
and the \emph{entropy dissipation inequality}
\begin{equation}
\label{eq:entropy}
\partial_t \entropy(u)
\le - \entropyproduction(u)
.
\end{equation}
For functions $u \in L_+^\infty(\Omega)$ such that $\Phi \in L^2(0, T; 
H^1(\Omega))$ we understand \eqref{eq:energy} and \eqref{eq:entropy} in the 
sense of measures, i.~e., that for any smooth nonnegative compactly supported 
function $\chi \colon (0, T) \to \mathbb R$ we respectively have
\begin{gather*}
%\label{eq:energy-meas}
- \int_0^T \chi'(t) \energy(u) \cd t
\le \iint_{Q_T}
\chi(t)\big(- | \nabla \Phi |^2
+ (\Phi_x + uf_x) \cdot \nabla \Phi
+ uf\Phi
\big)
\cd x \cd t
,
\\
\int_0^T \chi'(t) \entropy(u) \cd t
\ge
\int_0^T \chi(t) \entropyproduction(u) \cd t
.
\end{gather*}
If~\eqref{eq:entropy} holds in the sense of measures, the derivative 
$\partial_t \entropy(u)$ is a nonpositive distribution and hence a measure, 
while the entropy $\entropy(u)$ itself a.~e.\ coincides with a non-increasing 
function.

An important question is whether the entropy can be controlled by the entropy 
production, since this would imply the exponential stability of the 
equilibrium.  It turns out that this is true
provided that the $L^1$-norm of $u$ is bounded away from~$0$.  Specifically, we 
have
\begin{theorem}[Entropy-entropy production inequality]
\label{th:eep}
Suppose that~$f$ satisfies~\eqref{eq:f1a}--\eqref{eq:f1d} as well 
as~\eqref{eq:large-u}.  Let $U \subset L_+^\infty(\Omega)$ be a set of 
functions such that for any $u \in U$, we have $\Phi(\cdot, u(\cdot)) \in 
H^1(\Omega)$ and
\begin{equation}
\label{eq:eep-a}
\inf_{u \in U} \| u \|_{L^1(\Omega)} > 0
.
\end{equation}
Then there exists $C_U$ such that
\begin{equation}
\label{eq:eep-b}
\entropy(u) \le C_U \entropyproduction(u) \quad (u \in U)
.
\end{equation}
\end{theorem}
Theorem~\ref{th:eep} is a consequence of a fairly general functional 
inequality established in Section~\ref{ss:s2}.

\begin{theorem}[Existence of weak solutions]
\label{th:ex}
Suppose that~$f$ satisfies~\eqref{eq:f1a}--\eqref{eq:f1d} as well 
as~\eqref{eq:large-u} and~\eqref{eq:small-u}.  Then for any~$u^0 \in 
L^\infty_+(\Omega)$ there exists a nonnegative weak solution $u \in 
L^\infty(\Omega \times (0, \infty))$ of problem \eqref{eq:p1}--\eqref{eq:p3} 
enjoying the following properties:
\begin{enumerate}
\item (upper $L^\infty$-bound)
\begin{equation}
\label{eq:linfty}
\| u \|_{L^\infty(\Omega \times (0, \infty))}
\le
\inf
\left\{
\xi \ge 0 \colon \sup_{x \in \Omega} f(x, \xi)
\le - \esssup_{x \in \Omega} f^-(x, u^0(x))
\right\}
;
\end{equation}
\item $u$ satisfies the energy inequality~\eqref{eq:energy} in the sense of 
measures and
\begin{equation}
\label{eq:init-energy}
\esslimsup_{t \to +0} \energy(u(t)) \le \energy(u^0);
\end{equation}
\item
$u$ satisfies the entropy dissipation inequality~\eqref{eq:entropy} in the 
sense of measures and
\begin{equation}
\label{eq:init-entropy}
\esssup_{t > 0} \entropy(u(t)) \le \entropy(u^0);
\end{equation}

\item (lower $L^1$-bound)
\begin{equation}
\label{eq:lower}
\| u(t) \|_{L^1(\Omega)} \ge \| \min(u^0, m) \|_{L^1(\Omega)} \quad 
\text{a.~a.\ } t > 0
.
\end{equation}
\end{enumerate}
\end{theorem}

\begin{remark}
Theorem \ref{th:ex}, mutatis mutandis, is also valid in the case of the pure 
Fokker-Planck equation \eqref{eq:nfp}. Even in this case, our conditions on 
the nonlinearity $f$ are more relaxed than the ones available in the 
literature, see, e.g., \cite{AL83,BH86,K90A,K90B,FK95,Vaz07,CJM01,Bar16} and 
the references therein.
\end{remark}

\begin{remark}
In the general case, uniqueness of solutions cannot be expected due to the 
non-Lipschitz reaction term. However, our weak solutions are unique provided 
the initial data is bounded away from zero, see Theorem~\ref{lem:sp}. 
\end{remark}

\begin{remark}
Under the hypotheses of Theorem~\ref{th:ex}, the right-hand side 
of~\eqref{eq:linfty} is always finite (see Remark~\ref{rem:bnd}).  Moreover, 
if $u^0$ satisfies an estimate $\| u^0 \|_{L^\infty(\Omega)} \le a$, 
inequality~\eqref{eq:linfty} provides an estimate $\| u 
\|_{L^\infty(\Omega \times (0, \infty))} \le C_a$.
\end{remark}

The next theorem shows that the solutions that we have constructed 
exponentially converge to $m$. Note that \eqref{eq:small-u}
is not needed for the long-time convergence.
\begin{theorem}[Convergence to equilibrium]
\label{th:convergence}
Assume \eqref{eq:large-u} and suppose that a weak solution~$u$ 
of~\eqref{eq:p1}--\eqref{eq:p3} with the initial data $u^0\not\equiv 0$ 
satisfies the entropy dissipation inequality~\eqref{eq:entropy}, 
inequality~\eqref{eq:init-entropy}, and the lower $L^1$-bound~\eqref{eq:lower}. 
Then~$u$ exponentially converges to~$m$ in the sense of entropy:
\begin{equation}
\label{eq:convergence-a}
\entropy(u(t)) \le \entropy(u^0) \mathrm{e}^{-\gamma t} \quad \text{a.~a.\ } t 
> 0
,
\end{equation}
where $\gamma > 0$ can be chosen uniformly over initial data satisfying
\begin{equation}
\label{eq:convergence-b}
\| \min(u^0, m) \|_{L^1(\Omega)} \ge c
\end{equation}
with some $c > 0$.
\end{theorem}

Theorems~\ref{th:eep}, \ref{th:ex}, and \ref{th:convergence} are proved in 
Section~\ref{ss:nid}.

\subsection{Motivation and background} \label{ss:mb}  

The nonlinear Fokker-Planck equation \begin{equation} \label{eq:nfp}
\partial_t u = -\Div (u \nabla (f(x,u))) \end{equation} is intended to express 
the behaviour of stochastic systems coming from various branches of physics, 
chemistry and biology, see \cite{F05,Ts09,J16,BLMV}. In order to take into 
account the creation and annihilation of mass, the general 
drift-diffusion-reaction equation \eqref{eq:p1} was suggested in \cite{F04}. 
In the considerations of \cite{F04} (cf. also \cite{F05}), the crucial role is 
played by the free energy functional that up to an additive constant 
coincides with our relative entropy functional $\mathcal E$ from 
\eqref{eq:entr}. We opt for this change of terminology (though for 
thermodynamists the free energy involves the (physical) entropy, the internal 
energy, and the temperature) because in mathematical analysis it is convenient 
to refer to the basic Lyapunov functional of a system as the entropy, cf. 
\cite[p. 270]{villani03topics}.

On the other hand, equation \eqref{eq:p1} is a general nonlinear model for the 
spatial dynamics of a population that is tending to achieve the \emph{ideal 
free distribution} \cite{FC69,fr72} (the distribution that happens if 
everybody is free to choose its location) in a heterogeneous environment. The 
dispersal strategy is determined by a local intrinsic characteristic of 
organisms called \emph{fitness} (see, e.g., \cite{cosner05,cos13}). The 
fitness manifests itself as a growth rate, and simultaneously affects the 
dispersal as the species move along its gradient towards the most favorable 
environment. In \eqref{eq:p1}, $u(x,t)$ is the density of organisms, and 
$f(x,u)$ is the fitness.  The equilibrium $u(x)\equiv m(x)$ when the fitness 
is constantly zero corresponds to the ideal free distribution.  The original 
model \cite{mc90,cosner05} assumes a linear logistic fitness \begin{equation} 
f=m(x)-u \label{eq:lfit} \end{equation} but in general it can be any nonlinear 
function of the spatial variable and the density, cf. \cite{cos13}. The 
assumptions \eqref{eq:f2}, \eqref{eq:f3}, \eqref{eq:f4} are natural as they 
simply mean that the fitness is decreasing with respect to the population 
density (as the resources are limited), being positive for very small 
densities and negative for very large densities. Our Theorem 
\ref{th:convergence} indicates that the populations converge to the ideal free 
distribution with an exponential rate. 

The existence of weak solutions for the fitness-driven dispersal 
model~\eqref{eq:p1}--\eqref{eq:p3} with the logistic fitness \eqref{eq:lfit} 
was shown in  \cite{CW13}, and the entropic exponential convergence to $m$ was 
established in \cite{KMV16A}. The same kind of results for cross-diffusion 
systems involving several interacting populations (with logistic fitnesses) 
can be found in \cite{KMV16B}. Related two-species models were investigated in 
\cite{ccl13,ltw14}, where one population uses the fitness-driven dispersal 
strategy and the other diffuses freely or does not move at all. A system of two 
interacting populations with a particular nonlinear fitness function has 
recently been considered in \cite{XBF17}, which is the only existing 
mathematical treatment of a non-logistic fitness model that we are aware of.

But perhaps our main motivation to study \eqref{eq:p1} is that it is a 
gradient flow of the entropy functional $\mathcal{E}$ with respect to the 
intriguing recently introduced distance on the space of Radon measures, which 
is related to the \emph{unbalanced} optimal transport (i.e., failing to 
preserve the total transported mass), and that is referred to as the 
Hellinger-Kantorovich distance or the Wasserstein-Fisher-Rao distance 
\cite{KMV16A,peyre_1_2015,LMS_big_2015,LMS16,CPSV18}. This distance endows   the set 
of Radon measures with a formal (infinite dimensional) Riemannian metric 
$\langle\cdot,\cdot\rangle$, and provides first- and second-order differential 
calculus \cite{KMV16A} in the spirit of Otto 
\cite{otto01,villani03topics,villani08oldnew}. In particular, one can compute 
the metric gradients of the functionals of the form $$\mathcal 
F(u)=\int_\Omega F(x,u(x))\cd x$$ by the formula
\begin{equation}
\label{eq:formula_grad_d_scalar}
\grad \mathcal F(u)=-\Div\left(u\nabla \frac{\delta F}{\delta 
u}\right)+u\frac{\delta F}{\delta u},
\end{equation}
where $\frac{\delta F}{\delta u}=\partial_uF(x,u)$ stands for the first 
variation with respect to $u$ and $\nabla=\nabla_x$ is the usual gradient in 
space.  We refer to \cite{KMV16A} for further details and explanations. Since 
$f=-\partial_u E$, we can recast \eqref{eq:p1} as a gradient flow 
\begin{equation}
\partial_t u=-\grad \mathcal{E}(u).
\label{eq:gradf}
\end{equation}
The entropy dissipation identity \eqref{eq:entropy-diss-class}, which by the 
way was already known to Frank \cite{F04}, is then nothing but the archetypal 
property of  gradient flows
\begin{equation*}
\frac d {dt}\mathcal{E}(u)=-\langle\grad \mathcal{E}(u),\grad \mathcal{E}(u) 
\rangle_u.
\end{equation*}

In this connection, we recall that for the metric gradient flows like \eqref{eq:gradf}, the geodesic convexity of the driving entropy functional (or at least semi-convexity, i.e., $\lambda$-convexity with a negative constant $\lambda$) makes a difference \cite{otto01,AGS06,villani03topics,villani08oldnew,San17}. 
The presence of convexity allows one to apply minimizing movement schemes 
\cite{AGS06,JKO} to construct solutions to the gradient flow.  Moreover, 
$\lambda$-convexity with $\lambda$ strictly positive enables the Bakry-Emery 
procedure that usually yields the exponential convergence of the relative 
entropy to zero. Minimizing movement schemes for Hellinger-Kantorovich 
gradient flows of geodesically convex functionals and for related 
reaction-diffusion equations were suggested in \cite{MG16,GLM17}.  

Our entropy $\mathcal{E}$ is geodesically $(-1/2)$-convex with respect to the Hellinger-Kantorovich structure if $f=1-u^\alpha$, $\alpha>0$, but fails to be semi-convex for $f=u^\alpha-1$, $\alpha<0,$ and for $f=-\log u$ (the latter option corresponds to the interesting case of the Boltzmann entropy).  The spatial heterogeneity further complicates the situation. The quadratic (logistic) multicomponent entropy considered in \cite{KMV16B,KMV17} is not even semi-convex. All this can be observed by computing the Hessian of the entropy, cf. \cite[Section 3.4]{KMV16A}; the non-convexity of the Boltzmann entropy with respect to the Hellinger-Kantorovich metric was also mentioned in \cite{MG16,GLM17,LMS_big_2015,LMS16}. We refer to \cite{KV19A} for a more detailed discussion of examples of $f$ and the corresponding geodesic non-convexity. However, Santambrogio \cite{San17} emphasizes that the lack of geodesic convexity is not a universal obstacle for the study of gradient flows; our results in  the current paper and in \cite{KMV16B,KMV17,KMV16A,KV19,KV19A,SV17} illustrate this idea. 

\section{Generalized dissipation inequalities}
\label{ss:s2}
\subsection{Setting}
\label{ss:statement}
Motivated by the expressions for the entropy and entropy production, we forget 
for a while problem~\eqref{eq:p1}--\eqref{eq:p3} and consider the integrals
\begin{gather}
\int_\Omega E(x, u(x)) \, \mathrm dx
,
\label{eq:int1}
\\
\int_\Omega \left( g(x, u(x)) + u | \nabla_x f(x, u(x)) |^p \right) \mathrm dx
\label{eq:int2}
\end{gather}
on their own right.  Here $\Omega$ a domain in  $\mathbb R^d$; $p \ge 1$; the 
functions
\begin{gather*}
E, g \colon \Omega \times (0, \infty) \to [0, \infty),
\\
f \colon \Omega \times (0, \infty) \to \mathbb R
\end{gather*}
are fixed, and $u$ varies over a set $U$ of functions $\Omega \to (0, \infty)$.  
Observe that the nonnegativity of $E$ and $g$ ensures the existence of the 
integrals~\eqref{eq:int1} and~\eqref{eq:int2}, although they need not be 
finite.

The functions~$f$ and~$E$ introduced in Section~\ref{ss:ee} are, of course, 
prototypes for the ones appearing in~\eqref{eq:int1} and~\eqref{eq:int2}, but 
we assume no formal relationship between them.  In particular, in this section 
we do not suppose that~$f$ satisfies~\eqref{eq:f1a}--\eqref{eq:small-u}.

We would like to know whether~\eqref{eq:int1} can be controlled 
by~\eqref{eq:int2} uniformly with respect to $u \in U$.  In general, this is 
not the case, cf. a related discussion in \cite{KV19}.  However, we show that under suitable assumptions on the 
functions $E$, $f$, and $g$, \eqref{eq:int2} does indeed 
control~\eqref{eq:int1} provided that the set $U$ of admissible $u$ is 
separated from $0$ in some sense. 

For simplicity, we concentrate on the regular case.  Section~\ref{ss:gs} 
contains a discussion of possible generalisations. 
\begin{theorem}
\label{th:yai}
Let $\Omega$ be a bounded, connected, open domain in $\mathbb R^d$ admitting 
the relative isoperimetric inequality.  Let $p \ge 1$.  Suppose that functions 
$E, g \in C(\Omega \times (0, 
\infty))$ and $f \in C^1(\Omega \times (0, 
\infty))$ satisfy
\begin{gather}
E \ge 0, \ g \ge 0;
\label{eq:Efg1}
\\
\lim_{\varepsilon \to 0} \sup_{\substack{0 < u \le \varepsilon \\ x \in 
\Omega}}
E(x, u) < \infty;
\label{eq:Efg2}
\\
\inf_{\substack{u > \varepsilon \\ x \in \Omega \\E(x,u) \ne 0}}
\frac{g(x, u)}{E(x, u)} > 0
\quad \forall \varepsilon > 0,
\label{eq:Efg3}
\\
\lim_{\varepsilon \to 0} \inf_{\substack{0 < u \le \varepsilon \\ x \in 
\Omega}} f(x, u)
>
\lim_{\varepsilon \to 0} \sup_{\substack{u > 0 \\ E(x, u) < \varepsilon}}
f(x, u)
\label{eq:Efg4}
.
\end{gather}
Finally, suppose that a set $U \subset C^1(\Omega)$ consisting of strictly positive
functions contains no sequence $\{u_n\}$ such that $\{E(\cdot, u_n(\cdot))\}$ is 
bounded in $L^1(\Omega)$ and $\{u_n\}$ converges to $0$ in measure.  Then 
there exists a constant $C = C(\Omega, p, E, g, f, U)$ such that
\begin{equation}
\label{eq:yai}
\int_\Omega
E(x, u(x)) \cd x
\le
C
\left(
\int_\Omega
\big( g(x, u(x))
+
u(x)| \nabla_x f(x, u(x))|^p
\big) \cd x
\right)
\quad
( u \in U)
.
\end{equation}
\end{theorem}
\begin{remark}
The isoperimetric inequality for $\Omega$ reads
\begin{equation}
P(A; \Omega) \ge c_\Omega |A|^{\frac{d-1}{d}},
\quad A \subset \Omega,\ |A| \le \frac 12 | \Omega |
,
\end{equation}
where $P(A; \Omega)$ denotes the relative perimeter of a Lebesgue measurable 
set $A$ of locally finite perimeter with respect to $\Omega$, cf.~\cite[Remark 
12.39]{Mag12}, \cite{Mazja}.  We recall that the relative perimeter is defined as
\begin{equation*}
P(A; \Omega) = |\mu_A| (\Omega)
,
\end{equation*}
where $\mu_A:=\nabla 1_A$ is the Gauss-Green measure associated with~$A$. The support of~$\mu_A$ is contained \cite{Mag12} in the 
topological boundary of~$A$.
\end{remark}
\begin{remark}
\label{rem:hyp}
If $E \in C(\overline \Omega \times \mathbb R_+)$, condition~\eqref{eq:Efg2} 
is automatically true. If the set $\{ (x, u) \in \overline \Omega \times 
\mathbb R_+ \colon E(x, u) = 0\}$ is compact, the right-hand side 
of~\eqref{eq:Efg4} is simplified to $\max_{E(x, u) = 0} f(x, u)$
and likewise, if $f \in C(\overline \Omega \times \mathbb R_+)$, the left-hand 
side of~\eqref{eq:Efg4} can be written as $\min_{x} f(x, 0)$.  As 
for~\eqref{eq:Efg3}, it is more tricky.  In Section~\ref{ss:gs} we show that 
it always holds in a particular setting relevant for gradient flows 
(Theorem~\ref{th:eepab}).
\end{remark}

\begin{remark} The infimum in \eqref{eq:Efg3} depends on $\varepsilon$ and may tend to zero as $\varepsilon\to 0$, otherwise the claim would be trivial. \end{remark}

\subsection{Strategy of the proof of Theorem~\ref{th:yai}}

Before starting the proof of Theorem~\ref{th:yai}, we would like to informally 
outline the underlying ideas.

For simplicity, we will opt for an argument by contradiction.  Of course, a direct proof 
could be presented (as we have recently done in \cite{KV19} for a related inequality), and a quantitative constant could be derived 
from it.  However, this would be much more cumbersome, 
and the constant obtained in this way would anyway not be optimal. 
Any discussion of quantitative 
constants lies beyond the scope of this article.

It easily follows from~\eqref{eq:Efg3} that~$g$ controls~$E$ from above unless~$u$ is 
small.  Moreover, we infer (Lemma \ref{lem:claim1}) that if the constant in \eqref{eq:yai} blows up, the sets where either~$u$ 
or~$E$ are small tend to grow and together occupy nearly all of $\Omega$, 
while the `transitional annulus'---where neither is small---collapses.  At 
this point we must be prepared to face the situation where the integral
\begin{equation}
\label{eq:uapprox0}
\int_{[u \ll 1]} E \cd x
\end{equation}
is controlled neither by $$\int_{[u \ll 1]} g \cd x$$ (because \eqref{eq:Efg3} is not applicable), nor by $$\int_{[E \ll 1]} g \cd x$$ (because $g$ may be small), nor by $$\int_{[u \gg 0, E \gg 0]} g \cd x$$ (because the 
`annulus' is too small).

This is where the term with the gradient comes into play.  The crucial 
observation is that the total variation of~$f$ over the `annulus' can be 
estimated from below.  Actually, condition ~\eqref{eq:Efg4} gives a universal 
lower bound on the variation of~$f$ between the `inner boundary' of the annulus 
(say, where~$u$ is small) and its `outer boundary' (where~$E$ is small).  All in all, the 
integral~\eqref{eq:uapprox0} is controlled by the area of the set $[ u\ll
1]$ (due to~\eqref{eq:Efg2}), which is controlled by the perimeter of this set 
(by the isoperimetric inequality), which is in turn controlled by the total 
variation of~$f$ over the `annulus'.  This eventually leads to a contradiction.  
Naturally, when this idea is implemented in Lemma~\ref{lem:coareaappl} and the subsequent 
reasoning, we must relate the total variation and the integral $\int_\Omega u 
|\nabla f|^p \cd x$.  Then we use the coarea formula and estimate the total 
variation of~$f$ by the perimeters of its superlevel sets.

\subsection{Proof of Theorem~\ref{th:yai}}
Here we prove Theorem~\ref{th:yai}.  We start with the following observations.

Under the hypotheses of Theorem~\ref{th:yai}, integral 
\eqref{eq:int1} is finite for $u \in U$ whenever so is
\begin{equation*}
\int_\Omega g(x, u(x)) \cd x
.
\end{equation*}
Indeed, according to~\eqref{eq:Efg2} we can choose $\varepsilon > 0$ such that
\begin{equation*}
A := \sup_{\substack{0 < u \le \varepsilon \\ x \in \Omega}}
E(x, u) < \infty
.
\end{equation*}
By~\eqref{eq:Efg3}, we have
\begin{equation*}
B := \inf_{\substack{u > \varepsilon \\ x \in \Omega \\E(x,u) \ne 0}}
\frac{g(x, u)}{E(x, u)} > 0
\end{equation*}
(possibly $B = \infty$).  Then $E(x, u) \le g(x, u)/B$ whenever $u > 
\varepsilon$, so
\begin{align*}
\int_\Omega E(x, u(x)) \cd x
& =
\int_{[u \le \varepsilon]} E(x, u(x)) \cd x
+
\int_{[u > \varepsilon]} E(x, u(x)) \cd x
\\
&
\le
A |\Omega| +
\frac 1B
\int_\Omega g(x, u(x)) \cd x
<\infty
,
\end{align*}
as claimed.

Take sequences $\{\varepsilon_n\}$ and $\{\xi_n\}$ such that $\varepsilon_n > 
0$, $\varepsilon_n \to 0$,
\begin{equation*}
0 < \xi_n \le \inf_{\substack{u > \varepsilon_n \\ x \in \Omega}} \frac{g(x, 
u)}{E(x, u)}
\end{equation*}
(this is possible according to~\eqref{eq:Efg3}), and $\xi_n \to 0$.

Assume that Theorem~\ref{th:yai} is not true.  Then there exists a sequence of 
functions $\{u_n\} \subset U$ such that
\begin{equation}
\label{eq:p01}
\int_\Omega
(
g_n
+
u_n | \nabla f_n |^p
)
\cd x
\le
\varepsilon_n \xi_{n}
\int_\Omega
E_n
\, \mathrm dx
,
\end{equation}
where
\begin{align*}
E_n(x) &= E(x, u_n(x)), \\
f_n(x) &= f(x, u_n(x)), \\
g_n(x) &= g(x, u_n(x))
.
\end{align*}
Clearly, $E_n, g_n \in C(\Omega)$ and $f_n \in C^1(\Omega)$.  Moreover, it 
easily follows from~\eqref{eq:Efg1}--\eqref{eq:Efg4} that
\begin{gather}
E_n(x) \ge 0, \ g_n(x) \ge 0;
\label{eq:Efg1n}
\\
\varlimsup_{n \to \infty} \sup_{[u_n \le \varepsilon_n]} E_n < \infty;
\label{eq:Efg2n}
\\
\varliminf_{n \to \infty} \inf_{[u_n \le \varepsilon_n]} f_n
>
\varlimsup_{\varepsilon \to 0} \sup_{[E_n < \varepsilon]}
f_n
\label{eq:Efg4n}
,
\end{gather}
and according to the choice of~$\xi_n$, we have
\begin{equation}
\label{eq:Efg3n}
g_n \ge \xi_n E_n \text{ on } [u_n > \varepsilon_n]
.
\end{equation}

We want to show that the sequence $\{E_n\}$ is bounded in $L^1(\Omega)$ and 
$u_n \to 0$ in measure, thus obtaining a contradiction.

We use~\eqref{eq:p01} to estimate
\begin{align*}
\frac{1}{\varepsilon_n}
\int_{\Omega}
u_n | \nabla f_n |^p
\, \mathrm dx
& \le
\xi_n
\int_\Omega
E_n
\, \mathrm dx
-
\frac{1}{\varepsilon_n}
\int_\Omega
g_n
\, \mathrm dx
\\
& \le
\xi_n
\int_\Omega
E_n
\, \mathrm dx
-
\frac{1}{\varepsilon_n}
\int_{[u_n > \varepsilon_n]}
g_n
\, \mathrm dx
\\
& \le
\xi_n
\int_\Omega
E_n
\, \mathrm dx
-
\frac{\xi_n}{\varepsilon_n}
\int_{[u_n > \varepsilon_n]}
E_n
\, \mathrm dx
\\
& =
- \xi_n (\varepsilon_n^{-1} - 1)
\int_{[u_n > \varepsilon_n]}
E_n
\, \mathrm dx
+ \xi_n
\int_{[u_n \le \varepsilon_n]}
E_n
\, \mathrm dx
.
\end{align*}
Thus, we have
\begin{equation}
\label{eq:p02}
\frac{1}{\varepsilon_n}
\int_\Omega
u_n | \nabla f_n |^p
\, \mathrm dx
\le
- \xi_n (\varepsilon_n^{-1} - 1)
\int_{[u_n > \varepsilon_n]}
E_n
\, \mathrm dx
+ \xi_n
\int_{[u_n \le \varepsilon_n]}
E_n
\, \mathrm dx
.
\end{equation}
For large $n$, the first term on the right-hand side is negative, so we 
conclude that
\begin{equation}
\label{eq:p03}
\frac{1}{\varepsilon_n}
\int
u_n | \nabla f_n |^p
\, \mathrm dx
\le
\xi_n \sup_{[u_n \le \varepsilon_n]} E_n | [ u_n \le \varepsilon_n ] |
.
\end{equation}

From~\eqref{eq:p02} we get
\begin{equation}
\label{eq:p031}
\int_{[u_n > \varepsilon_n]}
E_n
\, \mathrm dx
\le
\frac{1}{\varepsilon_n^{-1} - 1}
\int_{[u_n \le \varepsilon_n]}
E_n
\, \mathrm dx
\le
\frac{\sup_{[u_n \le \varepsilon_n]} E_n|[u_n \le 
\varepsilon_n]|}{\varepsilon_n^{-1} - 1}
\end{equation}
and by~\eqref{eq:Efg2n}, the last expression is bounded uniformly with respect 
to $n$.  Hence the sequence $\{E_n\}$ is bounded in $L^1(\Omega)$. 

\begin{lemma}
\label{lem:claim1}
Given $a > 0$,
\begin{equation}
\label{eq:claim1}
\lim_{n \to \infty}
|
[ u_n > \varepsilon_n ]
\cap
[ E_n > a ]
|
=0
.
\end{equation}
\end{lemma}
\begin{proof}
Using~\eqref{eq:p031}, we have:
\begin{align*}
\label{eq:claim101}
|
[ u_n > \varepsilon_n ]
\cap
[ E_n > a ]
|
& \le
\frac 1a
\int_{[u_n > \varepsilon_n]\cap[E_n > a]}
E_n
\, \mathrm dx
\\
& \le
\frac 1a
\int_{[u_n > \varepsilon_n]}
E_n
\, \mathrm dx
\\
& \le
\frac{|[u_n \le \varepsilon_n]|}{a(\varepsilon_n^{-1} - 1)}
\sup_{[u_n \le \varepsilon_n]} E_n
\to 0
\quad (n \to \infty)
,
\end{align*}
where we have taken into account~\eqref{eq:Efg2n}, so \eqref{eq:claim1} is 
proved.
\end{proof}

\begin{lemma}
\label{lem:claim2}
Given $a > 0$, for large $n$ we have
\begin{equation}
\label{eq:claim2}
| [ E_n > a  ] |
\le
2
| [ u_n \le \varepsilon_n  ] |
.
\end{equation}
\end{lemma}
\begin{proof}
Using the estimate
\begin{equation*}
|
[ u_n > \varepsilon_n ]
\cap
[ E_n > a ]
|
\le
\frac{|[u_n \le \varepsilon_n]|}{a(\varepsilon_n^{-1} - 1)}
\sup_{[u_n \le \varepsilon_n]} E_n
\end{equation*}
obtained in the proof of Lemma~\ref{lem:claim1}, we get
\begin{equation*}
| [ E_n > a  ] |
\le
| [ u_n \le \varepsilon_n  ] |
+
| [ u_n > \varepsilon_n  ] \cap [ E_n > a  ] |
\le
\left(
1 +
\frac{\sup_{[u_n \le \varepsilon_n]} E_n}{a(\varepsilon_n^{-1} - 1)}
\right)
| [ u_n \le \varepsilon_n  ] |
,
\end{equation*}
and the lemma follows.
\end{proof}

It follows from~\eqref{eq:Efg4n} that we can choose $a>0$, $\alpha$, and 
$\beta$, all independent of~$n$, such that for large $n$ we have
\begin{equation}
\label{eq:alphabeta}
\sup_{[E_n \le a]} f_n
\le \alpha < \beta \le
\inf_{[u_n \le \varepsilon_n]} f_n
.
\end{equation}

We can assume that the limit
\begin{equation*}
\lim_{n \to \infty} |[ u_n \le \varepsilon_n  ]|
\end{equation*}
exists.  It follows from~\eqref{eq:alphabeta} that for large $n$ the sets 
$[u_n \le \varepsilon_n]$ and $[E_n \le a]$ are disjoint, so in view of 
Lemma~\ref{lem:claim1} we have
\begin{equation}
\label{eq:u+E}
|
[ u_n \le \varepsilon_n ]
|
+
|
[ E_n \le a ]
|
\to |\Omega|
\end{equation}
Thus, we actually face three logical possibilities:
\begin{gather}
\lim_{n \to \infty} |[ u_n \le \varepsilon_n ]| = |\Omega|;
\label{eq:case-i}
\\
\lim_{n \to \infty} |[ u_n \le \varepsilon_n ]| = 0;
\label{eq:case-ii}
\\
\lim_{n \to \infty} |[ u_n \le \varepsilon_n ]| = \mu_0 \in (0, |\Omega|);
\label{eq:case-iii}
\end{gather}

As $\varepsilon_n \to 0$, \eqref{eq:case-i} clearly implies $u_n \to 0$ in 
measure, a contradiction.

In what follows we show that~\eqref{eq:case-ii} and~\eqref{eq:case-iii} are in 
fact impossible.  The following lemma is crucial.
\begin{lemma}
\label{lem:coareaappl}
We have
\begin{equation}
\label{eq:coareaappl}
\frac{1}{\varepsilon_n}
\int_{\Omega}
u_n | \nabla f_n |^p \rd x
\ge
\frac{1}{|[E_n > a] \cap [u_n > \varepsilon_n]|^{p-1}}
\left(
\int_\alpha^\beta P([f_n > t], \Omega) \cd t
\right)
^p
\end{equation}
\end{lemma}
\begin{proof}
We have
\begin{align}
\frac{1}{\varepsilon_n}
\int_{\Omega}
u_n | \nabla f_n |^p
\, \mathrm dx
& \ge
\int_{[E_n > a] \cap [u_n > \varepsilon_n]}
| \nabla f_n |^p
\, \mathrm dx
\\
& \ge
\frac{1}{|[E_n > a] \cap [u_n > \varepsilon_n] |^{p-1}}
\left(
\int_{[E_n > a] \cap [u_n > \varepsilon_n]}
| \nabla f_n |
\, \mathrm dx
\right)
^p
.
\label{eq:tmp-101}
\end{align}
Using the coarea formula, we get:
\begin{align}
\int_{[E_n > a] \cap [u_n > \varepsilon_n]}
| \nabla f_n |
\, \mathrm dx
& =
\int_{-\infty}^{\infty}
P([f_n > t]; [E_n > a] \cap [u_n > \varepsilon_n])
\, \mathrm dx
\notag
\\
& \ge
\int_{\alpha}^{\beta}
P([f_n > t]; [E_n > a] \cap [u_n > \varepsilon_n])
\, \mathrm dx
\label{eq:tmp-102}
\end{align}
Fix $t \in (\alpha, \beta)$.  Evoking the definition of the relative 
perimeter, we have
\begin{equation}
\label{eq:case2-1}
P([f_n > t]; [E_n > a] \cap [u_n > \varepsilon_n])
=
\left |\mu_{[f_n > t]} \right|([E_n > a] \cap [u_n > \varepsilon_n])
,
\end{equation}
where $\mu_{[f_n > t]}$ is the Gauss-Green measure.  Obviously, we have
\begin{equation*}
\supp \mu_{[f_n > t]} \cap \Omega \subset \partial_\Omega [f_n > t] \subset 
[f_n = t]
\end{equation*}
for any $t \in (\alpha, \beta)$.  It follows from~\eqref{eq:alphabeta} that
\begin{equation*}
[f_n = t] \subset [E_n > a] \cap [u_n > \varepsilon_n]
,
\end{equation*}
so
\begin{equation*}
\supp \mu_{[f_n > t]} \cap \Omega\subset [E_n > a] \cap [u_n > \varepsilon_n]
\end{equation*}
and continuing~\eqref{eq:case2-1}, we obtain
\begin{align*}
P([f_n > t]; [E_n > a] \cap [u_n > \varepsilon_n])
& =
\left|\mu_{[f_n > t]}\right|([E_n > a] \cap [u_n > \varepsilon_n])
\\
& =
\left|\mu_{[f_n > t]}\right|(\Omega)
\\
& =
P([f_n > t]; \Omega)
.
\end{align*}
Combining this with~\eqref{eq:tmp-101} and~\eqref{eq:tmp-102}, we 
obtain~\eqref{eq:coareaappl}.
\end{proof}

Let us show that~\eqref{eq:case-ii} is impossible.  Assume that it holds.

If at a point~$x$ we have $f_n(x) > t$, $t\in (\alpha,\beta)$, \eqref{eq:alphabeta} guarantees that $E_n(x) > a$. Hence, $[f_n > t] 
\subset [E_n > a]$.  It follows from~\eqref{eq:case-ii} and~\eqref{eq:u+E} that 
$[|E_n \le a]| \to |\Omega|$, and thus $|[E_n > a]| \to 0$, so we conclude that 
$|[f_n > t]|$ is uniformly in~$t$ small when~$n$ is large.  For such large~$n$ 
we can apply the isoperimetric inequality:
\begin{equation*}
P([f_n > t]; \Omega) \ge c_\Omega |[f_n > t]|^\frac{d-1}{d}
.
\end{equation*}

Now it follows from~\eqref{eq:alphabeta} that $[u_n \le \varepsilon_n] \subset 
[f_n > t]$, so we have
\begin{equation*}
P([f_n > t]; \Omega) \ge c_\Omega |[u_n \le \varepsilon_n]|^\frac{d-1}{d}
.
\end{equation*}

Plugging this estimate into~\eqref{eq:coareaappl}, we obtain
\begin{equation*}
\frac{1}{\varepsilon_n}
\int_{\Omega}
u_n | \nabla f_n |^p
\, \mathrm dx
\ge
\frac{c_\Omega^p (\beta - \alpha)^p|[u_n \le \varepsilon_n]|^{p(d-1)/d}}{|[E_n 
> a] \cap [u_n > \varepsilon_n]|^{p-1}}
.
\end{equation*}
Estimating
\begin{equation*}
|[E_n > a] \cap [u_n > \varepsilon_n]|
\le
|[E_n > a]|
\le 2
|[u_n \le \varepsilon_n]|
\end{equation*}
by virtue of~\eqref{eq:claim2}, we obtain
\begin{equation*}
\frac{1}{\varepsilon_n}
\int_{\Omega}
u_n | \nabla f_n |^p
\, \mathrm dx
\ge
\frac{c_\Omega^p (\beta - \alpha)^p|[u_n \le 
\varepsilon_n]|^{p(d-1)/d}}{2^{p-1}|[u_n \le \varepsilon_n ] |^{p-1}}
= C |[u_n \le \varepsilon]|^{1-p/d}
,
\end{equation*}
where $C$ is independent of $n$.

Combining obtained estimate with~\eqref{eq:p03}, we get:
\begin{equation*}
C
| [ u_n \le \varepsilon_n ] |^{1 - \frac{p}{d}}
\le
\xi_n \sup_{[u_n \le \varepsilon_n]} E_n | [ u_n \le \varepsilon_n ] |
,
\end{equation*}
whence
\begin{equation*}
C
\le
\xi_n \sup_{[u_n \le \varepsilon_n]} E_n
|[u_n \le \varepsilon_n]|^{\frac pd} \to 0
\quad (n \to \infty)
,
\end{equation*}
as $\xi_n \to 0$ and the suprema are bounded by~\eqref{eq:Efg2n}.  This 
contradicts the fact that the left-hand side is a positive constant 
independent of $n$.  Thus, \eqref{eq:case-ii} is impossible.

It remains to show that~\eqref{eq:case-iii} is also impossible.  Assume that 
it holds.

It is easy to check that in this case we have
\begin{equation}
\label{eq:c2a1}
P([f_n > t]; \Omega) \ge p_0 \quad(\alpha < t < \beta)
,
\end{equation}
where $p_0 > 0$ is independent of $t$ and $n$.  Indeed, we have the inclusions
\begin{equation*}
[u_n \le \varepsilon_n] \subset [f_n > t] \subset [E_n > a]
\end{equation*}
and as in our case the measure of the first and third terms goes to~$\mu_0$ as 
$n \to \infty$, we also have
\begin{equation*}
|[f_n > t]| \to \mu_0 \text{ uniformly in } t \in (\alpha, \beta)
.
\end{equation*}
Now it suffices to apply the isoperimetric equality to $[f_n > t]$ if $\mu_0 < 
1/2$ and to $[f_n \le t]$ otherwise.

Plugging~\eqref{eq:c2a1} into~\eqref{eq:coareaappl}, we get
\begin{equation*}
\frac{1}{\varepsilon_n}
\int_{\Omega} u_{n} |\nabla f_{n}|^p
\cd x  \ge
\frac{1}{|[u_{n} > \varepsilon_n] \cap [E_n > a]|^{p-1}}
(\beta - \alpha)^p p_0^p
.
\end{equation*}
Comparing this with~\eqref{eq:p03}, we obtain
\begin{equation*}
  \frac{1}{|[u_{n} > \varepsilon_n] \cap [E_n > a]|^{p-1}}
    (\beta - \alpha)^p p_0^p
  \le
\xi_n \sup_{[u_n \le \varepsilon_n]} E_n | [ u_n \le \varepsilon_n ] | \to 0 
\quad (n \to \infty)
  .
\end{equation*}
As $n \to \infty$, the left-hand side remains bounded away from~0, while the 
right-hand side goes to~0, a contradiction.

\subsection{Generalisations and specialisations}
\label{ss:gs}

We start with the remark that Theorem~\ref{th:yai} can often be applied if $U$ 
is a subset of a space~$X$ of functions defined on~$\Omega$ provided that 
$C^1(\Omega)$ is dense in~$X$ and the integrals~\eqref{eq:int1} 
and~\eqref{eq:int2} are continuous with respect to the topology of~$X$.  
Indeed, if $U_1 = U \cap C^1(\Omega)$ is dense in~$U$, we apply the theorem 
to~$U_1$ and proceed by density to make sure that the same constant works 
for~$U$ as well.  On the other hand, if $U_1$ is not dense in~$U$, we 
replace~$U$ with its small enlargement~$\widetilde U$ in the cone of 
nonnegative functions in~$X$ and apply the same reasoning to~$\widetilde U$.  A more complicated density 
argument is used in the proof of Theorem~\ref{th:eep} given in 
Section~\ref{ss:nid}.

Another question is whether the constant~$C$ can be chosen uniformly with 
respect to $(E, g, f)$ if the latter triple is allowed to vary 
over a set $\mathcal X$.  It turns out that Theorem~\ref{th:yai} can be easily 
extended to handle this case.  Specifically, if the suprema and infima 
in~\eqref{eq:Efg2}--\eqref{eq:Efg4} are additionally taken over $(E, g, f) \in 
\mathcal X$, the constant~$C$ can be chosen independently of~$(E, g, f)$.  The 
proof remains essentially the same.  Assuming the converse, we have violating 
sequences $\{(\widetilde E_n, \tilde g_n, \tilde f_n)\} \subset \mathcal X$ 
and $\{u_n\} \subset U$ such that~\eqref{eq:p01} holds with
\begin{align*}
E_n(x) &= \tilde E_n(x, u_n(x)), \\
f_n(x) &= \tilde f_n(x, u_n(x)), \\
g_n(x) &= \tilde g_n(x, u_n(x))
.
\end{align*}
Moreover, the functions $E_n$, $g_n$, and $f_n$ 
satisfy~\eqref{eq:Efg1n}--\eqref{eq:Efg3n}.  The rest of the proof can be reused 
verbatim.

It should also be noted that the bare~$u$ on the right-hand side 
of~\eqref{eq:yai} can be replaced by a nonnegative function~$v(x, u(x))$.  Of 
course, in this case it no longer makes sense to require that $U$ should 
consist exclusively of positive functions.  The separation from~$0$ should be 
taken in the sense that no sequence~$\{v(\cdot, u_n(\cdot))\}$, where $u_n \in 
U$ and the sequence $\{E_n(\cdot, u_n(\cdot))\}$ is bounded in $L^1(\Omega)$, 
converges to~$0$ in measure.  However, if~$v$ is, for example, an increasing 
function vanishing at~$0$, this new condition is clearly equivalent to the 
original one.

Again, the proof remains essentially unchanged, the sets $[u_n > 
\varepsilon_n]$ and $[u_n \le \varepsilon_n]$ being replaced by $[v_n > 
\varepsilon_n]$ and $[v_n \le \varepsilon_n]$, respectively (here $v_n(x) = 
v(x, u_n(x))$).

Summarising, we have the following strengthened version of 
Theorem~\ref{th:yai}:
\begin{theorem}
\label{th:very-strong}
Let $\Omega$ be a bounded, connected, open domain in $\mathbb R^d$ admitting 
the relative isoperimetric inequality.  Let $p \ge 1$ and $I$ be an interval 
(possibly unbounded).  Let $\mathcal X = \{(E, g, f, v)\}$ be a set of tuples 
such that $E, g, v \in C(\Omega \times I)$, $f \in C^1(\Omega \times I)$, and
\begin{gather}
E \ge 0,\ g \ge 0,\ v \ge 0 \quad \forall (E, g, f, v) \in \mathcal X;
\label{eq:sEfg1}
\\
\lim_{\varepsilon \to 0}
\sup\{ E(x, u) \colon
(E, f, g, v) \in \mathcal X,
(x, u) \in \Omega \times I,
v(x, u) \le \varepsilon \} < \infty
\label{eq:sEfg2}
\\
\inf
\left \{
\frac{g(x, u)}{E(x, u)}
\colon
(E, f, g, v) \in \mathcal X,
(x, u) \in \Omega \times I,
E(x, u) \ne 0,
v(x, u) > \varepsilon
\right \}
> 0 \quad \forall \varepsilon > 0
\label{eq:sEfg3}
\\
\lim_{\varepsilon \to 0}
\inf \{ f(x, u)
\colon (E, f, g, v) \in \mathcal X,
(x, u) \in \Omega \times I,
v(x, u) \le \varepsilon \}
\notag
\\
>
\lim_{\varepsilon \to 0}
\sup \{ f(x, u)
\colon (E, f, g, v) \in \mathcal X,
(x, u) \in \Omega \times I,
E(x, u) \le \varepsilon \}
\label{eq:sEfg4}
\end{gather}
Finally, suppose that a set~$U \subset C^1(\Omega; I)$ satisfies the following 
requirement: for any sequences $\{(E_n, g_n, f_n, v_n)\} \subset \mathcal X\}$ 
and $\{u_n\} \subset U$ such that the sequence $\{E_n(\cdot, u_n(\cdot))\}$ is 
bounded in~$L^1(\Omega)$, the sequence~$\{v_n(\cdot, u_n(\cdot))\}$ does not 
converge to $0$ in measure.  Then there exists a constant $C$ depending only 
on $\Omega$, $p$, $U$ and $\mathcal X$ such that
\begin{multline*}
\int_\Omega
E(x, u(x)) \cd x
\\
\le
C
\left(
\int_\Omega
\big(
g(x, u(x))
+
v(x, u(x))| \nabla_x f(x, u(x))|^p
\big) \cd x
\right)
\quad
((E, g, f, v) \in \mathcal X, u \in U)
.
\end{multline*}
\end{theorem}
The proof is left to the reader.

Another option would be to allow for nonnegative instead of strictly 
positive~$u$ in Theorem~\ref{th:yai}.  In this case one assumes that $E \in 
C(\Omega \times [0, \infty))$ and that the supremum in~\eqref{eq:Efg2} is taken 
over $0 \le u \le \varepsilon$ and $x \in \Omega$.  The resulting inequality 
differs from~\eqref{eq:yai} in that the integral on the right-hand side is 
taken over $[u > 0]$.  The only modification needed in the proof is that 
whenever~$g$ or~$u | \nabla f|^p$ are integrated over~$\Omega$, the domain of 
integration should be changed to~$[u > 0]$. Note that this does not fit into the previous theorem because $f$ can be undefined on $[u=0]$. 

We conclude by showing that Theorem~\ref{th:yai} is applicable in a situation 
relevant for gradient flows.  In the subsequent formulation, $f_u$ and $E_u$ 
denote the derivatives of the functions~$f$ and~$E$, respectively, with respect 
to their second argument.

\begin{theorem}
\label{th:eepab}
Suppose that functions $E \in C(\overline \Omega \times [0, \infty))$, $f 
\in C^1(\overline \Omega \times (0, \infty))$, and $m \in C(\overline \Omega)$ 
satisfy
\begin{gather}
E(x, u) \ge 0, \quad (x, u) \in \overline \Omega \times [0, \infty);
\label{eq:eep-E}
\\
m(x) > 0, \quad x \in \overline \Omega;
\label{eq:eep-m}
\\
E(x, m(x)) = 0, \quad x \in \Omega;
\label{eq:eep-Em}
\\
E_u(x, u) = -f(x, u), \quad (x, u) \in \Omega \times (0, \infty);
\label{eq:eep-Ef}
\\
f_u(x, u) < 0, \quad (x, u) \in \overline \Omega \times (0, \infty)
\label{eq:eep-f}
\end{gather}
and let $U \subset C^1(\Omega)$ be a set of strictly positive functions having 
the property that no sequence $\{u_n\} \subset U$ such that $\{E(\cdot, 
u_n(\cdot))\}$ is bounded in~$L^1(\Omega)$, converges to~$0$ in measure.  
Finally, let $\sigma \in (0, \min_{\overline \Omega} m)$ and
\begin{equation*}
v_\sigma(\xi) = \frac{\xi^2}{\max(\xi, \sigma)}
.
\end{equation*}
Then we have
\begin{equation}
\int_\Omega
E(x, u(x)) \cd x
\le C
\int_\Omega
v_\delta(u(x))
\big(
(f(x, u(x)))^2
+
|\nabla_x f(x, u(x))|^2
\big)
\cd x
\quad (u \in U)
,
\end{equation}
where~$C > 0$ depends on~$\Omega$, $f$, $\sigma$, and~$U$.
\end{theorem}
\begin{remark}
Observe that under the hypotheses of Theorem~\ref{th:eepab}, the functions~$E$ 
and~$m$ are uniquely determined by~$f$.  Indeed, if $x \in \Omega$ is fixed, 
$E(x, u)$ as a function of~$u$ attains its minimum at $m(x) > 0$, so $E_u(x, 
m(x)) = 0$, i.~e., $f(x, m(x)) = 0$, according to~\eqref{eq:eep-Ef}.  This 
uniquely defines~$m(x)$, as it follows from~\eqref{eq:eep-f} that~$f(x, u)$ 
strictly decreases with respect to~$u$.  Now, $E(x, u)$ is the antiderivative 
of~$-f(x, u)$ with respect to~$u$ vanishing at~$m(x)$.
\end{remark}
\begin{proof}
We check the hypotheses of Theorem~\ref{th:very-strong} with $I = (0, 
\infty)$, $p = 2$, $g(x, u) = v_\sigma(u)(f(x, u))^2$, and the set $\mathcal X$ 
consisting of the single tuple~$(E, g, f, v_\sigma)$.  Clearly, we 
have~\eqref{eq:sEfg1}, while~\eqref{eq:sEfg2}--\eqref{eq:sEfg4} are equivalent 
to~\eqref{eq:Efg2}--\eqref{eq:Efg4}.

Recalling Remark~\ref{rem:hyp}, we see that~\eqref{eq:Efg2} holds.

Let us check~\eqref{eq:Efg4}.  Fix $x \in \Omega$.  The function~$E(x, u)$ is 
strictly convex in~$u$ and attains its zero minimum only at $u = m(x)$.  As 
$f(x, m(x)) = 0$, we see that
\begin{equation*}
\lim_{\varepsilon \to 0} \sup_{E < \varepsilon} f = \max_{E = 0} f = 0.
\end{equation*}
On the other hand, as~$f$ decreases with respect to~$u$, we have
\begin{align*}
\lim_{\varepsilon \to 0}
\inf_{\substack{0 < u \le \varepsilon \\ x \in \Omega}} f(x, u)
& \ge
\inf_{x \in \Omega} f(x, \sigma)
\\
& = \inf_{x \in \Omega} \int_{\sigma}^{m(x)} (- f_u(x, u)) \cd u
\\
& \ge
\min_{\substack{\sigma \le u \le m(x) \\ x \in \overline \Omega}}
(- f_u(x, u))
\min_{x \in \overline \Omega} (m(x) - \sigma) > 0
,
\end{align*}
so~\eqref{eq:Efg4} indeed holds.

It remains to check~\eqref{eq:Efg3}.  Without loss of generality, assume that $\varepsilon > 0$ is such that
\begin{gather}
\varepsilon < \frac 12
\min_{x \in \overline \Omega}
(-2m(x) f_u(x, m(x))),
\label{eq:tmp-alpha1}
\\
\varepsilon < 
\min_{x \in \overline \Omega}
(- f_u(x, m(x))).
\label{eq:tmp-alpha2}
\end{gather}
By Cauchy's mean value theorem, for any $x \in \Omega$, $u > \sigma$, $u \ne 
m(x)$, we have
\begin{equation}
\label{eq:tmp-5}
\frac{g(x, u)}{E(x, u)}
=
\frac{g(x, u)- g(x, m(x))}{E(x, u)- E(x, m(x))}
=
\frac{g_u(x, \xi_{x,u})}{E_u(x, \xi_{x,u})}
= - f(x, \xi_{x, u}) - 2\xi_{x, u} f_u(x, \xi_{x, u})
,
\end{equation}
where $\xi_{x, u}$ is some point between~$u$ and~$m(x)$.

By uniform continuity, there exists $\delta \in (0, \min_{\overline \Omega} m 
- \sigma)$ such that
\begin{equation*}
| \xi - m(x) | < \delta
\end{equation*}
implies
\begin{gather}
| -f(x, \xi) - 2\xi f_u(x, \xi) + 2m(x) f_u(x, m(x)) | < \varepsilon,
\label{eq:tmp-1}
\\
| f_u(x, \xi) - f_u(x, m(x))| < \varepsilon
.
\label{eq:tmp-2}
\end{gather}
Then from~\eqref{eq:tmp-1} and~\eqref{eq:tmp-alpha1} we see that
\begin{equation}
\label{eq:tmp-3}
|\xi - m(x)| < \delta \Rightarrow -f(x, \xi) - 2\xi f_u(x, \xi) > \varepsilon
.
\end{equation}
Further, using~\eqref{eq:tmp-2} and~\eqref{eq:tmp-alpha2}, we have
\begin{equation*}
-f(x, m(x) + \delta) =
\int_{m(x)}^{m(x) + \delta} (-f_u(x, u) \cd u) \ge \varepsilon \delta
,
\end{equation*}
whence, recalling that $f_u$ is negative and~$f$ is decreasing, we conclude
\begin{equation}
\label{eq:tmp-4}
\xi \ge m(x) + \delta \Rightarrow -f(x, \xi) - 2 \xi f_u(x, \xi) > \varepsilon 
\delta
.
\end{equation}
Now, if $|u - m(x)| < \delta$, the point~$\xi_{x, u}$ also satisfies $|\xi - 
m(x)| < \delta$, so we use~\eqref{eq:tmp-3} to conclude from~\eqref{eq:tmp-5} 
that
\begin{equation}
\label{eq:tmp-6}
\frac{g(x, u)}{E(x, u)} > \varepsilon
.
\end{equation}
If $u \ge m(x) + \delta$, then either $m(x) < \xi_{x, u} < m(x) + \delta$ and 
we again obtain~\eqref{eq:tmp-6}, or $\xi_{x, u} \ge m(x) + \delta$ and then 
we use~\eqref{eq:tmp-4} to get
\begin{equation*}
\frac{g(x, u)}{E(x, u)} > \varepsilon \delta
.
\end{equation*}
Thus,
\begin{equation*}
\inf_{\substack{u > \varepsilon \\ x \in \Omega \\ E(x, u) \ne 0}}
\frac{g(x, u)}{E(x, u)}
\ge
\min
\left(
\min_{\substack{\varepsilon \le u \le m(x) - \delta \\ x \in \overline 
\Omega}}
\frac{g(x, u)}{E(x, u)}
,
\varepsilon
,
\varepsilon \delta
\right)
> 0
,
\end{equation*}
since the function $g/E$ is continuous and positive on the compact set
\begin{equation*}
\{(x, u) \colon x \in \overline \Omega, \varepsilon \le u \le m(x) - \delta \}
.
\end{equation*}
We have showed that~\eqref{eq:Efg3} holds.

Thus, the hypotheses of Theorem~\ref{th:very-strong} are fulfilled and the 
inequality follows.
\end{proof}

\section{Technicalities}
\subsection{Positive classical solutions}

Let
\begin{equation*}
\theta(s) =
\begin{cases}
1 & \text{if } s > 0, \\
0 & \text{if } s \le 0
\end{cases}
\end{equation*}
be the Heaviside step function.

\begin{lemma}
\label{lem:leones}
If nonnegative $u, \hat u \in C^\infty(\overline \Omega)$ satisfy the no-flux 
boundary condition \eqref{eq:p2}, then
\begin{equation}
\label{eq:leones}
\int_\Omega
\theta(u - \hat u) \Div(u \nabla f - \hat u \nabla \hat f) \cd x \ge 0
,
\end{equation}
where $f$ and $\hat f$ stand for $f(x, u(x))$ and $f(x, \hat u(x))$, 
respectively.
\end{lemma}
\begin{proof}
Without loss of generality, the functions $u$ and $\hat u$ are defined and 
smooth on $\mathbb{R}^d$. Consider the set $\Upsilon:= [u - \hat u > 0]$. 
First let us assume that $0$ is a regular value of the function $u - \hat u$,  
then the boundary of $\Upsilon$ is smooth. Employing de Giorgi's Gauss-Green 
formula \cite[Theorem 15.9]{Mag12} and the formula for the Gauss-Green measure 
of an intersection \cite[Theorem 16.3]{Mag12}, we compute
\begin{multline*}
\int_\Omega \theta(u - \hat u) \Div (u \nabla f - \hat u \nabla \hat f) \cd x
 = \int_{\Upsilon\cap\Omega} \Div (u \nabla f - \hat u \nabla \hat f) \cd x
\\ = \int_{\partial^*(\Upsilon\cap\Omega)} (u \nabla f - \hat u \nabla \hat f) \cdot \nu_{\Upsilon\cap\Omega} \cd \mathcal H^{d-1}
= \int_{\partial \Upsilon\cap\Omega} (u \nabla f - \hat u \nabla \hat f) \cdot \nu_{\Upsilon} \cd \mathcal H^{d-1} \\ +  \int_{\Upsilon\cap\partial \Omega} (u \nabla f - \hat u \nabla \hat f) \cdot \nu_{\Omega} \cd \mathcal H^{d-1} +  \int_{[\nu_{\Upsilon}=\nu_{\Omega}]} (u \nabla f - \hat u \nabla \hat f) \cdot \nu_{\Omega} \cd \mathcal H^{d-1},
\end{multline*}
where $\nu_{\Upsilon\cap\Omega}$ is the measure-theoretic outward unit normal vector along the reduced boundary $\partial^*(\Upsilon\cap\Omega)$ of the intersection \cite{Mag12}.  Due to the no-flux boundary condition, the last two integrals vanish. On~$\partial \Upsilon\cap\Omega$, we have $u = \hat u$ and consequently, $f = \hat f$.  Thus, we can 
write
\begin{equation}
\label{eq:leones-1}
\int_\Omega \theta(u - \hat u) \Div (u \nabla f - \hat u \nabla \hat f) \cd x
= \int_{\partial \Upsilon\cap\Omega}  u \nabla (f - \hat f) \cdot \nu_{\Upsilon} \cd \mathcal H^{d-1}
.
\end{equation}
Due to the monotonicity of~$f$, we have $\Upsilon = [f - \hat f < 
0]$.  We see then that whenever $\nabla (f - \hat f) \ne 0$ on~$\partial \Upsilon$, 
$\nabla(f - \hat f)$ is an outward normal vector along $\partial \Upsilon$.  Thus, $\nabla (f - \hat f) 
\cdot \nu_{\Upsilon} \ge 0$ and equality~\eqref{eq:leones-1} gives~\eqref{eq:leones}.

In the general case, take a decreasing sequence~$\varepsilon_n \to 0$ such 
that $0$ is a regular value of $u  - \hat u- \varepsilon_n$.  Set
\begin{equation*}
\hat u_n = \hat u + \varepsilon_n, \ \hat f_n = f(x,\hat  u_n(x)).
\end{equation*}
By the above, we have
\begin{equation}
\label{eq:leones-2}
\int_\Omega
\theta(u - \hat u_n) \Div(u \nabla f - \hat u_n \nabla \hat f_n) \cd x \ge 0
.
\end{equation}
As~$\theta$ is left-continuous, we have
\begin{equation*}
\theta(u - \hat u_n) \to \theta(u - \hat u) \text{ pointwise in } \Omega
;
\end{equation*}
moreover, it is clear that
\begin{equation*}
\hat f_n \to \hat f \text{ in } C^2(\overline \Omega)
.
\end{equation*}
Passing to the limit in~\eqref{eq:leones-2}, we obtain~\eqref{eq:leones}.
\end{proof}

\begin{lemma}[$L^1$-contraction for positive classical solutions]
\label{lem:l1c}
Let $u$ and $\hat u$ be classical solutions of~\eqref{eq:p1}--\eqref{eq:p3} on 
$[0, T]$ with different initial data.  Suppose that $u$ and $\hat u$ satisfy
\begin{gather*}
\kappa \le u \le \frac{1}{\kappa}, \ \kappa \le \hat u \le \frac{1}{\kappa} 
\quad \text{in } Q_T
\end{gather*}
with some $\kappa > 0$ and let $L_\kappa > 0$ be such that
\begin{equation}
\label{eq:l1c-1}
|u_1 f(x, u_1) - u_2 f(x, u_2) | \le L_\kappa | u_1 - u_2 | \quad x \in 
\overline \Omega, \ \forall u_1, u_2 \in \left(\kappa, \frac{1}{\kappa}
\right).
\end{equation}
Then for a.~a.\ $t > 0$,
\begin{equation}
\label{eq:l1c-2}
\partial_t \int_\Omega (u - \hat u)^+ \cd x
\le L_\kappa \int_\Omega (u - \hat u)^+ \cd x
.
\end{equation}
\end{lemma}
\begin{proof}
We have:
\begin{align*}
\partial_t \int_\Omega (u - \hat u)^+ \cd x
& = \int_\Omega \theta(u - \hat u) (\partial_t u - \partial_t \hat u) \cd x
\\
& = - \int_\Omega \theta(u - \hat u) \Div(u \nabla f - \hat u \nabla \hat f) 
\cd x
\\
& + \int_\Omega \theta(u - \hat u)(uf - \hat u \hat f) \cd x =: - I_1 + I_2
,
\end{align*}
where $f$ and $\hat f$ stand for $f(x, u(x, t))$ and $f(x, \hat u(x, t))$, 
respectively.  By Lemma~\ref{lem:leones}, we have $I_1 \ge 0$.  To 
estimate~$I_2$, we use~\eqref{eq:l1c-1} and the observation that the integrand 
vanishes where $u - \hat u < 0$, thus obtaining
\begin{equation*}
I_2 \le L_\kappa \int_\Omega (u - \hat u)^+ \cd x
.
\end{equation*}
Inequality~\eqref{eq:l1c-2} follows.
\end{proof}

For $c \in \mathbb R$, define $u_c \in C^2(\overline \Omega)$ by
\begin{equation}
\label{eq:uc0}
f(x, u_c(x)) = c
.
\end{equation}
As $f$ is monotonous in~$u$, we see that the function~$u_c$ is unique, but it 
does not need to exist for a given~$c$.  Note that $u_0 = m$.

\begin{remark}
\label{rem:norm-uc}
There is a simple formula for the $L^\infty$-norm of~$u_c$:
\begin{equation}
\label{eq:uc}
\| u_c \|_{L^\infty(\Omega)}
=
\inf
\left\{
\xi \ge 0 \colon \sup_{x \in \Omega} f(x, \xi)
\le c
\right\}
.
\end{equation}
It follows from the fact that due to the monotonicity of~$f$, the inequality 
$\xi \ge \|u_c\|_{L^\infty(\Omega)}$ or, equivalently, $\xi \ge u_c(x)$ for 
all $x \in \Omega$, holds if and only if $f(x, \xi) \le f(x, u_c(x)) \equiv c$ 
for all $x \in \Omega$, i.~e., when
\begin{equation*}
\label{eq:bnd-1}
\sup_{x \in \Omega} f(x, \xi) \le c
.
\end{equation*}
\end{remark}
\begin{remark}
\label{rem:uc}
\emph{If~\eqref{eq:large-u} holds, for any $u \in L^\infty_+(\Omega)$ the 
function~$u_c$ with
\begin{equation}
\label{eq:uc-1}
c = - \esssup_{x \in \overline \Omega} f^-(x, u(x))
\end{equation}
is well-defined and $u \le u_c$ a.~e.\ in~$\Omega$.}  Indeed, if the second 
alternative in~\eqref{eq:large-u} holds, for any $x \in \overline \Omega$, the 
function $f(x, \xi)$ assumes all the values in the interval $(-\infty, 0]$ 
as~$\xi$ varies over $[m(x), \infty)$; in particular, $f(x, \xi)$ attains the 
value~$c$.  If, on the other hand, the first alternative in~\eqref{eq:large-u} 
holds, take $\xi_1 \ge \| u\|_{L^\infty}$ such that $c_1 := f(x, \xi_1)$ is 
independent of~$x$ and negative.  Clearly, for any fixed $x \in \overline 
\Omega$, the function $f(x, \xi)$ takes all the values in the interval $[c_1, 
0]$ as $\xi$ varies over $[m(x), \xi_1]$.  Now it suffices to observe that due 
to the monotonicity of~$f$, we have $c \in [c_1, 0]$.  One can prove in the 
same way that \emph{if~\eqref{eq:small-u} holds, for any function~$u$ 
essentially bounded away from~$0$ on~$\Omega$, there exists $u_c$ such that $u 
\ge u_c$ a.~e.\ in~$\Omega$, and $c\ge0$.}
\end{remark}
\begin{remark}
\label{rem:bnd}
It follows from Remarks~\ref{rem:uc} and~\ref{rem:norm-uc} that 
if~\eqref{eq:large-u} holds, the right-hand side of~\eqref{eq:linfty} is 
finite for any $u \in L^\infty_+(\Omega)$.
\end{remark}

\begin{lemma}[Restricted $L^1$-contraction]
\label{lem:l1cr}
Let $u$ be a classical solution of \eqref{eq:p1}--\eqref{eq:p3} on~$[0, 
\infty)$.  Then for $c \le 0$ we have
\begin{equation}
\label{eq:lc1}
\int_\Omega (u - u_c)^+ \cd x
\le
\int_\Omega (u^0 - u_c)^+ \cd x
,
\quad t > 0
\end{equation}
and likewise, for $c \ge 0$ we have
\begin{equation}
\label{eq:lc2}
\int_\Omega (u - u_c)^- \cd x
\le
\int_\Omega (u^0 - u_c)^- \cd x
,
\quad t > 0
\end{equation}
provided that $u_c$ exists.
\end{lemma}
\begin{proof}
Let us prove~\eqref{eq:lc1} for $c \le 0$.  Computing the derivative of the 
left-hand side, for a.~a.~$t > 0$ we get
\begin{align*}
\partial_t
\int_\Omega(u - u_c)^+ \cd x
& = \int_\Omega \theta(u - u_c) \partial_t u \cd x
\\
& = - \int_\Omega \theta(u - u_c) \Div(u\nabla f)  \cd x
\\
& + \int_\Omega \theta(u - u_c) uf \cd x =: -I_1 + I_2
.
\end{align*}
As $\nabla f(x, u_c(x)) \equiv 0$, we can use Lemma~\ref{lem:leones} to get 
$I_1 \ge 0$.  Now, the integrand of~$I_2$ can only be non-zero where $u > 
u_c$, in which case $f \le c \le 0$ due to the monotonicity of~$f$; 
consequently, $I_2 \le 0$.  Thus, we have
\begin{equation*}
\partial_t
\int_\Omega(u - u_c)^+ \cd x
\le 0
\end{equation*}
and~\eqref{eq:lc1} follows.  Inequality~\eqref{eq:lc2} is proved in much the 
same way.
\end{proof}

\begin{lemma}
\label{lem:ex-smooth}
Suppose that~$f$ satisfies~\eqref{eq:large-u} and~\eqref{eq:small-u}.  Then for 
any smooth $u^0 \colon \overline \Omega \to (0, \infty)$ satisfying the 
non-flux boundary condition, problem \eqref{eq:p1}--\eqref{eq:p3} has a 
classical solution.
\end{lemma}
\begin{proof}
Equation~\eqref{eq:p1} can be cast in the form
\begin{equation*}
\partial_t u = - uf_u \Delta u - \nabla u \cdot (f_x + f_u \nabla u)
- u (f_{xx} + 2f_{xu} \cdot \nabla u + f_{uu} | \nabla u |^2 - f)
.
\end{equation*}
If we show that a classical solution is \emph{a priori} bounded and stays away 
from~0, we can ignore the fact that the coefficient $- u f_u$ can be 
degenerate or singular at $u = 0, \infty$ and infer the existence of the 
solution from the classical theory of quasilinear parabolic equations.

Indeed, according to Remark~\ref{rem:uc}, we can find $u_{c_1}$ and $u_{c_2}$ 
such that $c_2 \le 0 \le c_1$ and
\begin{equation*}
u_{c_1}(x) \le u^0(x) \le u_{c_2}(x) \quad (x \in \Omega)
.
\end{equation*}
Then it follows from Lemma~\ref{lem:l1cr} that
\begin{equation*}
u_{c_1}(x) \le u(x, t) \le u_{c_2}(x, t) \quad (x, t) \in \Omega \times (0, 
\infty)
,
\end{equation*}
providing the required bounds.
\end{proof}

\subsection{Positive initial data}

If the initial data~\eqref{eq:p3} is bounded away from~$0$, we approximate it 
with smooth functions and prove the existence and uniqueness of weak solutions 
to~\eqref{eq:p1}--\eqref{eq:p3} stated in Theorem~\ref{lem:sp} below.

\begin{lemma}
\label{lem:energy}
Suppose that $u \in L_+^\infty(Q_T)$ satisfies the energy 
inequality~\eqref{eq:energy} in the sense of measures; then
\begin{gather}
\| \energy(u) \|_{L^\infty(0,T)} \le \esslimsup_{t \to +0} \energy(u(t)) + CT,
\label{eq:nrg1}
\\
\| \nabla \Phi(\cdot, u(\cdot)) \|_{L^2(Q_T)} \le 2(\esslimsup_{t \to +0} 
\energy(u(t)) + CT),
\label{eq:nrg2}
\end{gather}
where $C > 0$ is determined by an upper bound on $\|u\|_{L^{\infty}(\Omega)}$.
\end{lemma}
\begin{proof}
The function
\begin{equation*}
t \mapsto
\energy(u(t)) - \int_0^t
\left(
- \int_{\Omega} | \nabla \Phi |^2 \cd x
+ \int_{\Omega} (\Phi_x + uf_x) \cdot \nabla \Phi \cd x
+ \int_{\Omega} uf\Phi \cd x \right) \rd t
\end{equation*}
has a non-positive derivative in the sense of measures, so it a.~e.\ coincides 
with a non-increasing function.  In other words, for a.~a.\ $t_0, t_1 \in (0, 
T)$, $t_0 < t_1$, we have
\begin{equation*}
\energy(u(t_1)) - \energy(u(t_0))
- \int_{t_0}^{t_1}
\left(
- \int_{\Omega} | \nabla \Phi |^2 \cd x
+ \int_{\Omega} (\Phi_x + uf_x) \cdot \nabla \Phi \cd x
+ \int_{\Omega} uf\Phi \cd x \right) \rd t
\le 0
.
\end{equation*}

An upper bound on $\|u\|_{L^\infty(Q_T)}$ defines essential upper bounds on 
$uf$, $\Phi = \Phi(x, u(x, t))$, $\Phi_x$, and $uf_x$, so for a.~a.\ $t \in 
(t_0, t_1)$ we can estimate
\begin{multline*}
\int_{\Omega} (\Phi_x + uf_x) \cdot \nabla \Phi \cd x
+ \int_{\Omega} uf\Phi \cd x
\\
\le
\frac 12 \int_{\Omega} | \nabla \Phi |^2 \cd x
+ \frac 12 \int_{\Omega} | \Phi_x + uf_x |^2 \cd x
+ \int_{\Omega} uf\Phi \cd x
\\
\le
\frac 12 \int_{\Omega} | \nabla \Phi |^2 \cd x
+C
,
\end{multline*}
whence
\begin{equation*}
\energy(u(t_1))
+ \frac 12 \int_{t_0}^{t_1} \int_{\Omega} | \nabla \Phi |^2 \cd x \cd t
\le
\energy(u(t_0)) + C(t_1 - t_0)
.
\end{equation*}
Passing to the essential upper limit as $t_0 \to 0$ and estimating $t_1 - t_0 
\le T$, we obtain
\begin{equation*}
\energy(u(t_1))
+ \frac 12 \int_{0}^{t_1} \int_{\Omega} | \nabla \Phi |^2 \cd x \cd t
\le
\esslimsup_{t \to +0} \energy(u(t)) + CT
,
\end{equation*}
whence \eqref{eq:nrg1} and \eqref{eq:nrg2} follow.
\end{proof}

\begin{theorem}[Solvability for positive data]
\label{lem:sp}
Suppose that~$f$ satisfies~\eqref{eq:f1a}--\eqref{eq:f1d} as well 
as~\eqref{eq:large-u} and~\eqref{eq:small-u}.  Then for any $u^0 \in 
L_+^\infty$ such that
\begin{equation*}
\kappa \le u^0 \le \frac 1 \kappa \quad \text{a.~e.\ in } \Omega
\end{equation*}
with some constant $\kappa > 0$, there exists a unique weak solution $$u \in 
L_+^\infty(\Omega \times [0, \infty)) \cap C([0, 
\infty); L^1(\Omega))$$ satisfying the following properties: i) the upper 
bound~\eqref{eq:linfty} and lower bound~\eqref{eq:lower}; ii) the energy and 
entropy dissipation inequalities as well as~\eqref{eq:init-energy} and 
\eqref{eq:init-entropy}; iii) the restricted contraction
\begin{gather}
\int_\Omega (u - u_c)^+ \cd x
\le \int_\Omega (u^0 - u_c)^+ \cd x \quad (c \le 0)
\label{eq:sp-d}
,
\\
\int_\Omega (u - u_c)^- \cd x
\le \int_\Omega (u^0 - u_c)^- \cd x \quad (c \ge 0)
\label{eq:sp-e}
\end{gather}
whenever $u_c$ is defined; iv) if $\hat u$ is another such solution with the 
initial data $\hat u^0$, the $L^1$-contraction holds:
\begin{equation}
\label{eq:sp-c}
\| (u(t) - \hat u(t))^+ \|_{L_1(\Omega)}
\le
\mathrm{e}^{L_\kappa t}
\| (u^0 - \hat u^0)^+ \|_{L_1(\Omega)}
,
\end{equation}
where $L_\kappa$ is defined by~\eqref{eq:l1c-1}.
\end{theorem}
\begin{proof}
Let $\{u_n^0\}$ be a sequence of smooth functions satisfying the no-flux 
boundary condition such that
\begin{equation}
\label{eq:sp1}
\kappa \le u^0_n(x) \le \frac 1 \kappa \quad \text{in } \Omega
\end{equation}
and
\begin{equation}
\label{eq:sp2}
u^0_n \to u^0 \quad \text{in $L^1(\Omega)$ and a.~e.\ in $\Omega$}
.
\end{equation}
Let $u_n$ be the classical solution of~\eqref{eq:p1}--\eqref{eq:p3} on $[0, 
\infty)$ with the initial data $u^0_n$.  For any $T > 0$, it follows from 
Lemma~\ref{lem:l1c} that
\begin{equation*}
\| u_n - u_m \|_{C([0,T]; L^1(\Omega))}
\le \mathrm e^{L_\kappa T} \| u^0_n - u^0_m \|_{L^1}
,
\end{equation*}
so $\{u_n\}$ is a Cauchy sequence in $C([0, T]; L^1(\Omega))$.  As $T$ is 
arbitrary, we see that $\{u_n\}$ converges in $C([0, \infty); L^1(\Omega))$ to 
some function~$u$.  We claim that it is the sought-for solution.

By Remark~\ref{rem:uc}, there exists $u_c$ ($c \le 0$) such that $u_c \ge 
1/\kappa$; then $u_c$ dominates the initial data $u^0_n$ and thus, the 
solutions~$u_n$ as well, which follows from Lemma~\ref{lem:l1cr}.  
Consequently, the sequence~$\{u_n\}$ is bounded in~$L^\infty(\Omega \times (0, 
\infty))$, so it converges to~$u$ weakly* in this space, whence $u \in 
L^\infty(\Omega \times (0, \infty))$.

Put
\begin{align*}
f_n & = f(x, u_n(x, t)), & f_{xn} & = f_x(x, u_n(x, t)), \\
\Phi_n & = \Phi(x, u_n(x, t)), & \Phi_{xn} & = \Phi_x(x, u_n(x, t)), \\
\Psi_n & = \Psi(x, u_n(x, t)), & E_n & = E(x, u_n(x, t)).
\end{align*}
Fix $T > 0$.  As the sequence $\{u_n\}$ is bounded in $L^\infty(Q_T)$, so are 
the sequences $\{u_n f_n\}$, $\{u_n f_{xn}\}$, $\{\Phi_n\}$, $\{\Phi_{xn}\}$, 
$\{\Psi_n\}$, and $\{E_n\}$.  Thus, there is no loss of generality in assuming
\begin{equation}
\label{eq:sp-conv}
\left.
\begin{aligned}
u_n & \to u \\
u_n f_n & \to uf \\
u_n f_{xn} & \to uf_x \\
\Phi_n & \to \Phi \\
\Phi_{xn} & \to \Phi_x \\
\Psi_n & \to \Psi
\\
E_n & \to E
\end{aligned}
\right\}
\begin{tabular}{l}
a.~e.\ in  $Q_T$, \\
strongly in any $L^p(Q_T)$, $1 \le p < \infty$, \\
weakly* in $L^\infty(Q_T)$, \\
and in the sense of distributions,
\end{tabular}
\end{equation}
where we write $\Phi$ for $\Phi(\cdot, u(\cdot))$, etc.  It follows 
from~\eqref{eq:sp-conv} that $\nabla \Phi_n \to \nabla \Phi$ in the sense of 
distributions.  The approximate solutions satisfy the energy inequality 
and~\eqref{eq:init-energy} while their initial energy is bounded, so we see 
from~\eqref{eq:nrg2} that the sequence $\nabla \Phi_n$ is bounded in 
$L^2(Q_T)$.  Consequently, $\Phi \in L^2(0, T; H^1(\Omega))$ and
\begin{equation}
\label{eq:sp-conv-a}
\nabla \Phi_n \to \nabla \Phi \quad \text{weakly in } L^2(Q_T)
.
\end{equation}

Let us check that $u$ is a weak solution of \eqref{eq:p1}--\eqref{eq:p3} on 
$[0, T]$.  Take an admissible test function~$\varphi$.  Writing the weak 
setting for the approximate solution, we have
\begin{equation}
\label{eq:sp3}
\int_0^T
\int_\Omega
(u_n \partial_t \varphi + ( - \nabla \Phi_n + \Phi_{xn} + u_n f_{xn}) \cdot 
\nabla \varphi + f_n u_n \varphi) \cd x \cd t
=
\int_\Omega u^0_n(x) \varphi(x, 0) \cd x
.
\end{equation}
It follows from~\eqref{eq:sp2}, \eqref{eq:sp-conv}, and \eqref{eq:sp-conv-a} 
that we can pass to the limit in~\eqref{eq:sp3} and obtain~\eqref{eq:def1} 
for~$u$.  Thus, $u$ is indeed a weak solution.

Let us show that $u$ satisfies the energy inequality on $[0, T]$ in the sense 
of measures.  Taking a smooth nonnegative test function $\varphi \in C^\infty$ 
vanishing outside of $[0, T]$, we write the energy inequality in the sense of 
measures for the approximate solutions:
\begin{multline*}
- \iint_{Q_T} \Psi_n \varphi'(t) \cd x \cd t
\le
- \iint_{Q_T} |\nabla \Phi_n|^2 \varphi(t) \cd x \cd t
\\
+ \iint_{Q_T} \varphi(t) (\Phi_{xn} + u_n f_{xn}) \cdot \nabla \Phi_n \cd x 
\cd t
+ \iint_{Q_T} u_n f_n \Phi_n \varphi(t) \cd x cd t
\end{multline*}
Convergences~\eqref{eq:sp-conv} ensure that we can pass to the limit in all 
the terms but for the first one on the right-hand side.  As for the latter, it 
follows from~\eqref{eq:sp-conv-a} that $\sqrt \varphi \, \nabla \Phi_n \to 
\sqrt \varphi \, \nabla \Phi$ weakly in $L^2(Q_T)$, whence
\begin{equation*}
\iint_{Q_T} \varphi | \nabla \Phi |^2 \cd x \cd t
\le
\liminf_{n \to \infty}
\iint_{Q_T} \varphi | \nabla \Phi_n |^2 \cd x \cd t
,
\end{equation*}
and the energy inequality follows.

Let us check~\eqref{eq:init-energy}.  The approximate solutions satisfy
\begin{equation*}
\esslimsup_{t \to +0} \energy(u_n(t)) \le \energy(u_n^0)
\end{equation*}
so by virtue of~\eqref{eq:nrg1} we obtain
\begin{equation*}
\esssup_{t \in (0, \varepsilon)} \energy(u_n(t)) \le \energy(u_n^0) + 
C\varepsilon
.
\end{equation*}
It follows from~\eqref{eq:sp2} and \eqref{eq:sp-conv} that
\begin{gather*}
\energy(u_n) \to \energy(u) \quad \text{weakly* in } L^\infty(0, \varepsilon)
,
\\
\energy(u^0_n) \to \energy(u^0)
,
\end{gather*}
so we get
\begin{align*}
\esssup_{t \in (0, \varepsilon)} \energy(u(t))
& \le
\liminf_{n \to \infty} \esssup_{t \in (0, \varepsilon)} \energy(u_n(t))
\\
& \le
\lim_{n \to \infty} \energy(u_n^0) + C\varepsilon
\\
& =\energy(u^0) + C\varepsilon
.
\end{align*}
Now sending $\varepsilon \to 0$ we recover~\eqref{eq:init-energy}.

Let us show that $u$ satisfies the entropy dissipation inequality on $[0, T]$ 
in the sense of measures.  Let $\varphi \in C^\infty$ be a smooth nonnegative 
test function vanishing outside of $[0, T]$.  The approximate solutions 
satisfy the entropy dissipation inequality in the sense of measures, so we 
have
\begin{equation*}
- \iint_{Q_T} E_n \varphi'(t) \cd x \cd t
\le
- \iint_{Q_T} \varphi(t) u_n f_n^2 \cd x \cd t
- \iint_{u_n > 0} \frac{\varphi(t)}{u_n} | - \nabla \Phi_n + \Phi_{xn} + u_n 
f_{xn}|^2 \cd x \cd t
.
\end{equation*}
Consequently, for any $\delta > 0$ we have
\begin{multline}
\label{eq:sp4}
- \iint_{Q_T} E_n \varphi'(t) \cd x \cd t
\le
- \iint_{Q_T} \frac{\varphi(t)}{\max(u_n, \delta)} (u_n f_n)^2 \cd x \cd t
\\
- \iint_{Q_T} \frac{\varphi(t)}{\max(u_n, \delta)}
| - \nabla \Phi_n + \Phi_{xn} + u_n f_{xn}|^2 \cd x \cd t
.
\end{multline}
Observe that
\begin{gather}
\label{eq:mery}
\frac{\varphi(t)}{\max(u_n, \delta)}
\to
\frac{\varphi(t)}{\max(u, \delta)}
\begin{tabular}{l}
a.~e.\ in  $Q_T$, \\
strongly in any $L^p$, $1 \le p < \infty$, \\
and weakly* in $L^\infty(Q_T)$,
\end{tabular}
\\ \label{eq:vk1}
v_n:=- \nabla \Phi_n + \Phi_{xn} + u_n f_{xn}
\to - \nabla \Phi + \Phi_{x} + u f_{x} \quad \text{weakly in 
$L^2(\Omega)$}
\end{gather}
We claim that \begin{multline} \iint_{Q_T} \frac{\varphi(t)}{\max(u, \delta)}
| - \nabla \Phi + \Phi_{x} + u f_{x}|^2 \cd x \cd t
\\ \leq \liminf_{n\to \infty}
\iint_{Q_T} \frac{\varphi(t)}{\max(u_n, \delta)}
| - \nabla \Phi_n + \Phi_{xn} + u_n f_{xn}|^2 \cd x \cd t
.\label{f:leon}\end{multline} Then, taking into account~\eqref{eq:sp-conv}, we can pass to the 
limit in~\eqref{eq:sp4} obtaining
\begin{multline*}
- \iint_{Q_T} E \varphi'(t) \cd x \cd t
\le
- \iint_{Q_T} \frac{\varphi(t)}{\max(u, \delta)} (u f)^2 \cd x \cd t
\\
- \iint_{Q_T} \frac{\varphi(t)}{\max(u, \delta)}
| - \nabla \Phi + \Phi_{x} + u f_{x}|^2 \cd x \cd t
.
\end{multline*}
On the set $\{(x, t) \in Q_T \colon u(x, t) = 0\}$ we have $uf_x = 0$ (by 
virtue of~\eqref{eq:f1d}), $\Phi_x = 0$ and $\Phi = 0$, whence also $\nabla 
\Phi = 0$ a.~e.\ on this set.  Thus, we can write
\begin{multline*}
- \iint_{Q_T} E \varphi'(t) \cd x \cd t
\le
- \iint_{Q_T} \frac{\varphi(t)}{\max(u, \delta)} (u f)^2 \cd x \cd t
\\
- \iint_{u > 0} \frac{\varphi(t)}{\max(u, \delta)}
| - \nabla \Phi + \Phi_{x} + u f_{x}|^2 \cd x \cd t
\end{multline*}
Letting $\delta \to 0$, by Beppo Levi's theorem we obtain the energy 
inequality.

To prove the technical claim \eqref{f:leon}, we use a variant of the Banach-Alaoglu theorem in varying $L^2(\rd\mu^n)$ spaces:
\begin{lemma}
\label{lem:variant_banach_alaoglu_vector_fields}
Let $\mathcal O\subset \R^N$ be an open set, $\mu_n$ a sequence of finite non-negative Radon measures narrowly converging to $\mu$, and $v_n$ a sequence of vector fields on $\mathcal O$.
If
$$
\|v_n\|_{L^2(\mathcal O,\rd\mu_n)}\leq C,
$$
then there exists $v\in L^2(\mathcal O,\rd\mu)$ such that, up to extraction of some subsequence,
\begin{equation} \label{leo1}
\forall \,\boldsymbol\zeta\in\mathcal C^\infty_c(\mathcal O):\qquad \lim\limits_{n\to\infty}\int_{\mathcal O}v_n\cdot \boldsymbol\zeta\, \rd \mu_n=\int_{\mathcal O}v\cdot \boldsymbol\zeta\, \rd \mu
\end{equation}
and
\begin{equation} \label{leo2}
\|v\|_{L^2(\mathcal O,\rd\mu)}\leq \liminf\limits_{n\to\infty}\| v_n\|_{L^2(\mathcal O,\rd\mu_n)}.
\end{equation} \end{lemma}

The proof of this fact by optimal transport techniques can be found in \cite{AGS06}; this lemma also follows from a variant of the Banach-Alaoglu theorem \cite[Proposition 5.3]{KMV16A}.
We will apply this lemma with $\mathcal O=Q_T$, $v_n$ from \eqref{eq:vk1}, and the sequence of measures $\rd\mu_n(t,x):=\frac{\varphi(t)}{\max(u_n, \delta)}\cd x\cd t$, which converges narrowly to $\rd\mu(t,x)=\frac{\varphi(t)}{\max(u, \delta)}\cd x\cd t$ due to the strong convergence \eqref{eq:mery}.
Extracting a subsequence if needed, we see that there is a vector-field $v\in L^2(\mathcal O,\rd\mu)$ verifying \eqref{leo1} and \eqref{leo2}.
On the other hand, by \eqref{eq:mery} and \eqref{eq:vk1},  $$v_n \frac{\varphi(t)}{\max(u_n, \delta)}\to (- \nabla \Phi + \Phi_{x} + u f_{x}) \frac{\varphi(t)}{\max(u, \delta)}$$ weakly in $L^1(Q_T)$.
Evoking \eqref{leo1}, we find that 
$$
\int_{\mathcal O} v \cdot \boldsymbol\zeta \cd\mu	=\int_{\mathcal O} (- \nabla \Phi + \Phi_{x} + u f_{x})\cdot \boldsymbol\zeta\cd\mu 
$$ for all test functions $\boldsymbol \zeta$.
By density, we conclude that $v=- \nabla \Phi + \Phi_{x} + u f_{x}$ in $L^2(\mathcal O,\rd\mu)$, and \eqref{f:leon} follows from \eqref{leo2}.

Inequality~\eqref{eq:init-entropy} is proved in the same way 
as~\eqref{eq:init-energy} given that it holds for the approximate solutions.

Inequalities \eqref{eq:sp-d}--\eqref{eq:sp-c} follow from correspondent 
inequalities for approximate solutions (Lemmas~\ref{lem:l1c} 
and~\ref{lem:l1cr}), as we obviously have
\begin{gather*}
\left.
\begin{aligned}
(u_n(t) - u_c)^{\pm} & \to (u(t) - u_c)^{\pm}
\\
(u_n(t) - \hat u_n(t))^+ & \to (u(t) - \hat u(t))^+
\end{aligned}
\right\}
\quad \text{in } L^1(\Omega),\ \forall t \ge 0
,
\end{gather*}
where the approximations $\hat u_n$ are constructed in the same way as $u_n$.

Contraction~\eqref{eq:sp-c} implies the uniqueness of~$u$.

To obtain the upper bound~\eqref{eq:linfty}, we define $c \le 0$ 
by~\eqref{eq:uc-1} and thus have $u^0 \le u_c$ on~$\Omega$, whence in view of 
contraction~\eqref{eq:sp-d},
\begin{equation*}
u(x, t) \le u_c(x), \quad (x, t) \in \Omega \times (0, \infty)
.
\end{equation*}
Recalling the formula~\eqref{eq:uc} for the norm of~$u_c$, we obtain the upper 
bound.

To obtain the lower $L^1$-bound~\eqref{eq:lower}, we take $u_c = m$ 
in~\eqref{eq:sp-e}, obtaining
\begin{align*}
\| u(t) \|_{L^1(\Omega)}
& \ge
\| \min(u(t), m) \|_{L^1(\Omega)}
= \int_\Omega (m - (u(t) - m)^-) \cd x
\\
& \ge \int_\Omega (m - (u^0 - m)^-) \cd x
= \| \min(u(t), m) \|_{L^1(\Omega)}
,
\end{align*}
as required.
\end{proof}

\subsection{Nonnegative initial data}
\label{ss:nid}

If initial data~\eqref{eq:p3} is only nonnegative, we approximate it with 
positive functions and reuse the proof of Theorem~\ref{lem:sp} to establish the 
existence of solutions to~\eqref{eq:p1}--\eqref{eq:p3} as stated in 
Theorem~\ref{th:ex} (but not uniqueness, owing to the loss of contraction).

\begin{proof}[Proof of Theorem~\ref{th:ex}]
Take a decreasing sequence $\varepsilon_n \to 0$ and set
\begin{equation*}
u^0_n = u^0 + \varepsilon_n
.
\end{equation*}
By Theorem~\ref{lem:sp}, there exists a weak solution $u_n$ of 
\eqref{eq:p1}--\eqref{eq:p3} with the initial data~$u^0_n$.  
Contraction~\eqref{eq:sp-c} ensures the comparison principle for this sequence 
of solutions, whence $u_{n + 1} \le u_n$ a.~e.\ in $\Omega \times (0, 
\infty)$.  Consequently, there exists the monotone limit $u \in 
L^\infty(\Omega \times (0, \infty))$ and moreover, we obviously have the 
convergences~\eqref{eq:sp-conv}.  From this moment on, the proof copies that 
of Theorem~\ref{lem:sp} except that~\eqref{eq:sp-d} and~\eqref{eq:sp-e} hold 
almost everywhere rather then everywhere.
\end{proof}

We conclude by proving Theorems~\ref{th:eep} and~\ref{th:convergence}.

\begin{proof}[Proof of Theorem~\ref{th:eep}]
Let $D = \left\{(x,\Phi(x,u))\colon x\in \Omega, u>0 \right\}$ and consider the function $\Xi \colon 
D \to [0, \infty)$ implicitly defined by the equation
\begin{equation*}
\Phi(x, \Xi(x, \phi)) = \phi
.
\end{equation*}
As $\Phi$ is monotonous with respect to its second argument, $\Xi$ is uniquely 
defined.  Clearly, $\Xi$ is $C^2$.

Fix $u \in U$.  We claim that there exists a sequence of functions $\Phi_n \in 
C(\overline \Omega) \cap C^\infty(\Omega)$ such that
\begin{gather*}
(x, \Phi_n(x)) \in D \quad (x \in\Omega),
\\
\Phi_n \to \Phi(\cdot, u(\cdot)) \quad \text{in } H^1 \text{and a.~e.\ in } 
\Omega
\end{gather*}
Indeed, take a sequence $\{\delta_n\}$, where $\delta_n > 0$ and $\delta_n \to 
0$, put $\widetilde \Phi_n(x) = \Phi(x,u(x)) + \delta_n$, and let $\widetilde 
\Phi_n^\varepsilon$ be the mollification of~$\widetilde \Phi_n$.  Observe 
that~$\widetilde \Phi_n$ is strictly positive and so is $\widetilde 
\Phi_n^\varepsilon$. It suffices to show that for any $n$ sufficiently large there 
exists $\varepsilon_n > 0$ such that whenever $\varepsilon < \varepsilon_n$, 
we have
\begin{equation}
\label{eq:eep-11}
\{ (x, \widetilde \Phi_n^\varepsilon(x)) \colon x \in \Omega \} \subset D
.
\end{equation}

If the second alternative in~\eqref{eq:large-u} holds, for every $x\in \Omega$ we have \begin{equation*} \Phi(x,u)=\Phi(x,1) +\int_u^1 \xi f_u(x, \xi) \cd \xi\ge \Phi(x,1) +\int_u^1 f_u(x, \xi) \cd \xi=\Phi(x,1) + f(x,1)-f(x,u)\to +\infty\end{equation*} as $u\to +\infty$. This implies that $D = 
\Omega\times(0, \infty)$, so~\eqref{eq:eep-11} obviously holds with any~$\varepsilon$.

Assume the first alternative in~\eqref{eq:large-u}.  Take $\xi_0 \ge \|
u \|_{L^\infty(\Omega)}$ such that $f(x, \xi)$ does not depend on~$x$ if $\xi 
\ge \xi_0$ and set
\begin{equation*}
a = - \int_{\xi_0}^{\xi_0 + 1} u f_u(x, u) \cd x > 0
.
\end{equation*}
We have:
\begin{align*}
\Phi(x, \xi_0 + 1) - \widetilde \Phi_n(x)
& = \Phi(x, \xi_0 + 1) - \Phi(x, u(x)) - \delta_n
\\
& \ge \Phi(x, \xi_0 + 1) - \Phi(x, \xi_0) - \delta_n
= a - \delta_n
.
\end{align*}
Thus, for large $n$ we have

\begin{align*}
\widetilde \Phi_n(x) \leq \Phi(x, \xi_0 + 1) - \frac a 2
.
\end{align*} Upon mollification, \begin{align*}
\widetilde \Phi^\varepsilon_n(x) \leq \Phi^\varepsilon(x, \xi_0 + 1) - \frac a 2
.
\end{align*}
For a fixed $n$, the function $ \Phi(\cdot, \xi_0 + 1)$ is continuous on~$\overline 
\Omega$, so the mollifications $ \Phi^\varepsilon(\cdot, \xi_0 + 1)$ converge to 
it uniformly on $\overline \Omega$ as $\varepsilon \to 0$.  
Consequently, \begin{equation}
(x,\widetilde \Phi^\varepsilon_n(x) ) \in \{
(x, \phi) \in \Omega \times (0, \infty) \colon \phi \le \Phi(x, \xi_0 + 1)
\}
\subset D
\end{equation} for all $x\in \Omega$, proving~\eqref{eq:eep-11}.

Taking  a sequence $\{\Phi_n\}$ as above, we can set $u_n(x) = \Xi(x, 
\Phi_n(x))$, so that $\Phi_n(x) = \Phi(x, u_n(x))$.  Clearly, $u_n \in 
C^2(\Omega)$ and $u_n > 0$.  Further, the sequence $\{u_n\}$ is bounded in 
$L^\infty(\Omega)$ because so is $\{\Phi_n\}$, and due to the continuity of 
$\Xi$ we have
\begin{equation*}
u_n \to u \quad \text{a.~e.\ in } \Omega
.
\end{equation*}
As a consequence, for $f_n = f(x, u_n(x))$ and $E_n = E(x, u_n(x))$ we have
\begin{equation}
\label{eq:eep-conv}
\left.
\begin{aligned}
u_n & \to u \\
u_n f_n & \to uf \\
u_n f_{xn} & \to uf_x \\
%\Phi_n & \to \Phi \\
\Phi_{xn} & \to \Phi_x \\
E_n & \to E
\end{aligned}
\right\}
\begin{tabular}{l}
a.~e.\ in  $\Omega$ \\
and in any $L^p(\Omega)$, $1 \le p < \infty$
,
\end{tabular}
\end{equation}
where we write $f$ for $f(\cdot, u(\cdot))$, etc.  In particular, there is no 
loss of generality in assuming a lower bound
\begin{equation*}
\| u_n \|_{L^1(\Omega)} \ge c
:= \frac12 \inf_{u \in U} \| u \|_{L^1(\Omega)} > 0
\end{equation*}
(positivity by virtue of~\eqref{eq:eep-a}), where $c$ is obviously independent 
not only of~$u_n$ but of $u$ as well.

Define
\begin{equation*}
\widetilde U
= \{ w \in C^1(\Omega) \colon w > 0, \| w \|_{L^1(\Omega)} \ge c\}
.
\end{equation*}
By Theorem~\ref{th:eepab}, there exist a function
\begin{equation*}
v(\xi) = \frac{\xi^2}{\max(\xi, \sigma)},
\end{equation*}
where $\sigma > 0$, and a constant $C > 0$ such that
\begin{equation*}
\int_\Omega E(x, w(x)) \cd x
\le C
\int_\Omega v(w(x)) (f(x, w(x)) + | \nabla f(x, w(x)) |^2) \cd x
\quad (w \in \widetilde U)
.
\end{equation*}
In particular, as $u_n \in \widetilde U$, we see that
\begin{equation}
\label{eq:eep1}
\int_\Omega E_n \cd x
\le C
\int_\Omega v_n (f_n + | \nabla f_n |^2) \cd x
,
\end{equation}
where $v_n = v(u_n(x))$.

Let us check that we can pass to the limit in~\eqref{eq:eep1}.  First, it 
follows from~\eqref{eq:eep-conv} that
\begin{equation*}
\int_\Omega E_n \cd x \to \int_\Omega E \cd x
.
\end{equation*}
Next, note that we clearly have
\begin{equation*}
\frac{1}{\max(u_n, \sigma)}
\to
\frac{1}{\max(u, \sigma)}
\quad \text{a.~e.\ in $\Omega$ and weakly* in $L^\infty(\Omega)$}
\end{equation*}
and thus, again using~\eqref{eq:eep-conv}, we obtain
\begin{equation*}
\int_\Omega v_n f_n^2 \cd x
= \int_\Omega \frac{1}{\max(u_n, \sigma)} (u_n f_n)^2 \cd x
\to \int_\Omega \frac{1}{\max(u, \sigma)} (u f)^2 \cd x
.
\end{equation*}
Finally, as $u_n$ is smooth and positive, we can write
\begin{align*}
\int_\Omega v_n | \nabla f_n |^2 \cd x
& = \int_\Omega \frac{1}{\max(u_n, \sigma)} | - \nabla \Phi_n + \Phi_{xn} + 
u_n f_{xn} |^2 \cd x
\\
& \to \int_\Omega \frac{1}{\max(u, \sigma)} | - \nabla \Phi + \Phi_{x} + u 
f_{x} |^2 \cd x
.
\end{align*}
On the set $[u = 0]$ we have $uf_x = 0$ by \eqref{eq:f1d}, $\Phi_x = 0$, and 
$\Phi = 0$, the last equality implying $\nabla \Phi = 0$ a.~e.\ on $[u = 0]$.  
Thus, we can write
\begin{align*}
\int_\Omega v_n | \nabla f_n |^2 \cd x
\to \int_{[u > 0]} \frac{1}{\max(u, \sigma)} | - \nabla \Phi + \Phi_{x} + u 
f_{x} |^2 \cd x
.
\end{align*}
To sum up, we have
\begin{equation*}
\int_\Omega E \cd x
\le
C
\left(
\int_\Omega \frac{u^2}{\max(u, \sigma)} f^2 \cd x
+ \int_{[u > 0]} \frac{1}{\max(u, \sigma)} | - \nabla \Phi + \Phi_{x} + u 
f_{x} |^2 \cd x
\right)
,
\end{equation*}
which is even stronger than~\eqref{eq:eep-b}.
\end{proof}

\begin{proof}[Proof of Theorem~\ref{th:convergence}]
Let $U \subset L^\infty_+$ be the set of functions such that for any $v \in 
U$, we have $\Phi(\cdot, v(\cdot)) \in H^1(\Omega)$ and 
$\|v\|_{L^1(\Omega)} \ge c$.  By Theorem~\ref{th:eep} we have the 
entropy-entropy production inequality~\eqref{eq:eep-b} for~$U$.

Let $u$ be a weak solution of~\eqref{eq:p1}--\eqref{eq:p3} with the initial 
data satisfying~\eqref{eq:convergence-b}.  It follows from the lower 
$L^1$-bound~\eqref{eq:lower} that $u(t) \in U$ for a.~a.\ $t > 0$.  Combining 
the entropy dissipation and entropy-entropy production inequalities, we obtain
\begin{equation*}
\partial_t \entropy(u(t)) \le - C_U^{-1} \entropy(u(t)) \quad \text{a.~a. } t > 0
.
\end{equation*}
Letting $e(t) = \entropy(u(t)) \mathrm{e}^{C_U^{-1} t}$, we see that $\partial_t 
e(t) \le 0$ in the sense of measures, whence $e$ a.~e.\ coincides with a 
nonincreasing function.  Moreover,
\begin{equation*}
\esssup_{t > 0} e(t)
= \esslimsup_{t \to 0} e(t)
= \esslimsup_{t \to 0} \entropy(u(t)) \mathrm{e}^{C_U^{-1}t}
\le \entropy(u^0)
\end{equation*}
by virtue of~\eqref{eq:init-entropy}, so $e(t) \le \entropy(u^0)$ for a.~a.\ 
$t > 0$, yielding \eqref{eq:convergence-a} with $\gamma = C_U^{-1}$.
\end{proof}

\subsection*{Acknowledgment} The idea of this paper originated from conversations of the second author with Goro Akagi and Yann Brenier during a stay at ESI in Vienna. He would like to thank Goro Akagi and Yann Brenier for the inspiring discussions and correspondence, Ulisse Stefanelli for the invitation to the thematic program \emph{Nonlinear Flows} at ESI, and ESI for hospitality. The research was partially supported by the Portuguese Government through FCT/MCTES and by the ERDF through PT2020 (projects UID/MAT/00324/2019, PTDC/MAT-PUR/28686/2017 and TUBITAK/0005/2014).

\subsection*{Conflict of interest: none}


\begin{thebibliography}{10}

\bibitem{AL83}
H.~W. Alt and S.~Luckhaus.
\newblock Quasilinear elliptic-parabolic differential equations.
\newblock {\em Math. Z.}, 183(3):311--341, 1983.

\bibitem{AGS06}
L.~Ambrosio, N.~Gigli, and G.~Savar{\'e}.
\newblock {\em Gradient Flows: in Metric Spaces and in the Space of Probability
  Measures}.
\newblock Basel: Birkh{\"a}user Basel, 2008.

\bibitem{Bar16}
V.~Barbu.
\newblock Generalized solutions to nonlinear {F}okker-{P}lanck equations.
\newblock {\em J. Differential Equations}, 261(4):2446--2471, 2016.

\bibitem{BH86}
M.~Bertsch and D.~Hilhorst.
\newblock A density dependent diffusion equation in population dynamics:
  stabilization to equilibrium.
\newblock {\em SIAM J. Math. Anal.}, 17(4):863--883, 1986.

\bibitem{BLMV}
T.~Bodineau, J.~Lebowitz, C.~Mouhot, and C.~Villani.
\newblock Lyapunov functionals for boundary-driven nonlinear drift-diffusion
  equations.
\newblock {\em Nonlinearity}, 27(9):2111--2132, 2014.

\bibitem{ccl13}
R.~S. Cantrell, C.~Cosner, Y.~Lou, and C.~Xie.
\newblock Random dispersal versus fitness-dependent dispersal.
\newblock {\em J. Differential Equations}, 254(7):2905--2941, 2013.

\bibitem{CJM01}
J.~A. Carrillo, A.~J\"ungel, P.~A. Markowich, G.~Toscani, and A.~Unterreiter.
\newblock Entropy dissipation methods for degenerate parabolic problems and
  generalized {S}obolev inequalities.
\newblock {\em Monatsh. Math.}, 133(1):1--82, 2001.

\bibitem{peyre_1_2015}
L.~Chizat, G.~Peyr{\'e}, B.~Schmitzer, and F.-X. Vialard.
\newblock An interpolating distance between optimal transport and
  {F}isher--{R}ao metrics.
\newblock {\em Foundations of Computational Mathematics}, 18(1):1--44, 2018.

\bibitem{CPSV18}
L.~Chizat, G.~Peyr{\'e}, B.~Schmitzer, and F.-X. Vialard.
\newblock Unbalanced optimal transport: Dynamic and {K}antorovich formulations.
\newblock {\em Journal of Functional Analysis}, 274(11):3090--3123, 2018.

\bibitem{cosner05}
C.~Cosner.
\newblock A dynamic model for the ideal-free distribution as a partial
  differential equation.
\newblock {\em Theoretical Population Biology}, 67(2):101--108, 2005.

\bibitem{cos13}
C.~Cosner.
\newblock Beyond diffusion: conditional dispersal in ecological models.
\newblock In J.~{M}allet-{P}aret~et al., editor, {\em Infinite Dimensional
  Dynamical Systems}, pages 305--317. Springer, 2013.

\bibitem{CW13}
C.~Cosner and M.~Winkler.
\newblock Well-posedness and qualitative properties of a dynamical model for
  the ideal free distribution.
\newblock {\em Journal of mathematical biology}, 69(6-7):1343--1382, 2014.

\bibitem{FK95}
J.~Filo and J.~Ka\v{c}ur.
\newblock Local existence of general nonlinear parabolic systems.
\newblock {\em Nonlinear Anal.}, 24(11):1597--1618, 1995.

\bibitem{F04}
T.~D. Frank.
\newblock Asymptotic properties of nonlinear diffusion, nonlinear
  drift-diffusion, and nonlinear reaction-diffusion equations.
\newblock {\em Ann. Phys.}, 13(7-8):461--469, 2004.

\bibitem{F05}
T.~D. Frank.
\newblock {\em Nonlinear {F}okker-{P}lanck equations}.
\newblock Springer Series in Synergetics. Springer-Verlag, Berlin, 2005.
\newblock Fundamentals and applications.

\bibitem{fr72}
S.~D. Fretwell.
\newblock {\em Populations in a seasonal environment}.
\newblock Princeton University Press, 1972.

\bibitem{FC69}
S.~D. Fretwell and H.~L. Lucas.
\newblock On territorial behavior and other factors influencing habitat
  distribution in birds {I}. {T}heoretical development.
\newblock {\em Acta Biotheoretica}, 19(1):16--36, 1969.

\bibitem{GLM17}
T.~Gallou{\"e}t, M.~Laborde, and L.~Monsaingeon.
\newblock An unbalanced optimal transport splitting scheme for general
  advection-reaction-diffusion problems.
\newblock {\em arXiv:1704.04541}, 2017.

\bibitem{MG16}
T.~O. Gallou\"et and L.~Monsaingeon.
\newblock A {JKO} splitting scheme for {K}antorovich-{F}isher-{R}ao gradient
  flows.
\newblock {\em SIAM J. Math. Anal.}, 49(2):1100--1130, 2017.

\bibitem{JKO}
R.~Jordan, D.~Kinderlehrer, and F.~Otto.
\newblock The variational formulation of the {F}okker--{P}lanck equation.
\newblock {\em SIAM journal on mathematical analysis}, 29(1):1--17, 1998.

\bibitem{J16}
A.~J\"ungel.
\newblock {\em Entropy methods for diffusive partial differential equations}.
\newblock SpringerBriefs in Mathematics. Springer, [Cham], 2016.

\bibitem{K90A}
J.~Ka\v{c}ur.
\newblock On a solution of degenerate elliptic-parabolic systems in
  {O}rlicz-{S}obolev spaces. {I}.
\newblock {\em Math. Z.}, 203(1):153--171, 1990.

\bibitem{K90B}
J.~Ka\v{c}ur.
\newblock On a solution of degenerate elliptic-parabolic systems in
  {O}rlicz-{S}obolev spaces. {II}.
\newblock {\em Math. Z.}, 203(4):569--579, 1990.

\bibitem{KMV16B}
S.~Kondratyev, L.~Monsaingeon, and D.~Vorotnikov.
\newblock A fitness-driven cross-diffusion system from population dynamics as a
  gradient flow.
\newblock {\em J. Differential Equations}, 261(5):2784--2808, 2016.

\bibitem{KMV16A}
S.~Kondratyev, L.~Monsaingeon, and D.~Vorotnikov.
\newblock A new optimal transport distance on the space of finite {R}adon
  measures.
\newblock {\em Adv. Differential Equations}, 21(11-12):1117--1164, 2016.

\bibitem{KMV17}
S.~Kondratyev, L.~Monsaingeon, and D.~Vorotnikov.
\newblock A new multicomponent {P}oincar\'e-{B}eckner inequality.
\newblock {\em J. Funct. Anal.}, 272(8):3281--3310, 2017.

\bibitem{KV19}
S.~{Kondratyev} and D.~{Vorotnikov}.
\newblock {Convex Sobolev inequalities related to unbalanced optimal
  transport}.
\newblock {\em arXiv e-prints}, page arXiv:1904.04112, Apr 2019.

\bibitem{KV19A}
S.~Kondratyev and D.~Vorotnikov.
\newblock Spherical {H}ellinger-{K}antorovich gradient flows.
\newblock {\em SIAM J. Math. Anal.}, 51(3):2053--2084, 2019.

\bibitem{LMS16}
M.~Liero, A.~Mielke, and G.~Savar\'e.
\newblock Optimal transport in competition with reaction: the
  {H}ellinger-{K}antorovich distance and geodesic curves.
\newblock {\em SIAM J. Math. Anal.}, 48(4):2869--2911, 2016.

\bibitem{LMS_big_2015}
M.~Liero, A.~Mielke, and G.~Savar{\'e}.
\newblock Optimal entropy-transport problems and a new
  {H}ellinger--{K}antorovich distance between positive measures.
\newblock {\em Inventiones mathematicae}, 211(3):969--1117, 2018.

\bibitem{ltw14}
Y.~Lou, Y.~Tao, and M.~Winkler.
\newblock Approaching the ideal free distribution in two-species competition
  models with fitness-dependent dispersal.
\newblock {\em SIAM J. Math. Anal.}, 46(2):1228--1262, 2014.

\bibitem{mc90}
A.~D. MacCall.
\newblock {\em Dynamic geography of marine fish populations}.
\newblock Washington Sea Grant Program Seattle, 1990.

\bibitem{Mag12}
F.~Maggi.
\newblock {\em Sets of Finite Perimeter and Geometric Variational Problems: An
  Introduction to Geometric Measure Theory}.
\newblock Cambridge Studies in Advanced Mathematics. Cambridge University
  Press, 2012.

\bibitem{Mazja}
V.~G. Maz'ja.
\newblock {\em Sobolev spaces}.
\newblock Springer Series in Soviet Mathematics. Springer-Verlag, Berlin, 1985.
\newblock Translated from the Russian by T. O. Shaposhnikova.

\bibitem{otto01}
F.~Otto.
\newblock The geometry of dissipative evolution equations: the porous medium
  equation.
\newblock {\em Comm. Partial Differential Equations}, 26(1-2):101--174, 2001.

\bibitem{San17}
F.~Santambrogio.
\newblock \{{E}uclidean, metric, and {W}asserstein\} gradient flows: an
  overview.
\newblock {\em Bull. Math. Sci.}, 7(1):87--154, 2017.

\bibitem{SV17}
W.~Shi and D.~Vorotnikov.
\newblock The gradient flow of the potential energy on the space of arcs.
\newblock {\em Calculus of Variations and Partial Differential Equations, to
  appear}, 2019.

\bibitem{Ts09}
C.~Tsallis.
\newblock {\em Introduction to nonextensive statistical mechanics}.
\newblock Springer, 2009.

\bibitem{Vaz-igolki}
J.~L. V\'azquez.
\newblock Failure of the strong maximum principle in nonlinear diffusion.
  {E}xistence of needles.
\newblock {\em Comm. Partial Differential Equations}, 30(7-9):1263--1303, 2005.

\bibitem{Vaz07}
J.~L. V\'azquez.
\newblock {\em The porous medium equation}.
\newblock Oxford Mathematical Monographs. The Clarendon Press, Oxford
  University Press, Oxford, 2007.
\newblock Mathematical theory.

\bibitem{villani03topics}
C.~Villani.
\newblock {\em Topics in optimal transportation}.
\newblock American Mathematical Soc., 2003.

\bibitem{villani08oldnew}
C.~Villani.
\newblock {\em Optimal transport: old and new}.
\newblock Springer Science \& Business Media, 2008.

\bibitem{XBF17}
Q.~Xu, A.~Belmonte, R.~deForest, C.~Liu, and Z.~Tan.
\newblock Strong solutions and instability for the fitness gradient system in
  evolutionary games between two populations.
\newblock {\em J. Differential Equations}, 262(7):4021--4051, 2017.

\end{thebibliography}
\end{document}